\documentclass[aihp]{imsart}

\RequirePackage{amsthm,amsmath,amsfonts,amssymb}
\RequirePackage[numbers]{natbib}
\usepackage{bm}
\usepackage{enumitem}
\usepackage{microtype}
\usepackage[hidelinks]{hyperref}
\usepackage[nameinlink,capitalize]{cleveref}
\usepackage{mathtools}

\startlocaldefs
\newcommand{\R}{\mathbb R}

\newcommand{\E}{\mathbb E}
\newcommand{\Pp}{\mathbb P}
\newcommand{\cE}{\mathcal E}
\newcommand{\cF}{\mathcal F}
\newcommand{\cM}{\mathcal M}
\newcommand{\cD}{\mathcal D}
\newcommand{\cC}{\mathcal C}
\newcommand{\sgn}{\operatorname{sign}}
\newcommand{\Tr}{\operatorname{Tr}}
\newcommand{\supp}{\operatorname{supp}}

\newcommand{\vol}{\operatorname{vol}}
\newcommand{\eps}{\varepsilon}
\newcommand{\ip}[2]{\left\langle #1,#2\right\rangle}
\newcommand{\norm}[1]{\left\lVert #1\right\rVert}
\newcommand{\abs}[1]{\left|#1\right|}
\newcommand{\bone}{\bm{1}}

\theoremstyle{plain}
\newtheorem{theorem}{Theorem}[section]
\newtheorem{proposition}[theorem]{Proposition}
\newtheorem{lemma}[theorem]{Lemma}
\newtheorem{corollary}[theorem]{Corollary}
\theoremstyle{definition}

\newtheorem{assumption}[theorem]{Assumption}
\newtheorem{remark}[theorem]{Remark}

\endlocaldefs

\begin{document}

\begin{frontmatter}

\title{Fast Switching Near a Smooth Interface:
	A Curvature-Universal Kirchhoff Limit}
\runtitle{Curvature-Universal Switching Kirchhoff Limit}

\begin{aug}
\author[A]{\inits{LM}\fnms{Leonardo}~\snm{Marconi}\ead[label=e1]{leonardo.marconi5@unibo.it}}
\orcid{ORCID}
\address[A]{Università di Bologna, Dipartimento di Scienze Statistiche Paolo Fortunati\printead[presep={,\ }]{e1}}

\end{aug}

\begin{abstract}
	We derive a weighted Kirchhoff diffusion on a smooth Riemannian open book
	as a scaling limit of regime-switching Brownian motions on boundaryless
	manifolds.  The pages are compact Riemannian manifolds sharing a common
	boundary manifold and inducing the same metric on the binding, while their
	second fundamental forms and interior geometries may differ.
	
	For each $\varepsilon>0$, we regularize the metric in an $o(\varepsilon)$ collar, pass to smooth doubles of the pages, run Brownian motion on each double, and allow the page label to switch in a collar of width $\varepsilon$ at rates
	\(q_\varepsilon(r)Q_{ij},\)
	where $Q$ is an arbitrary irreducible finite-state Markov generator with invariant law $\pi$, while \(q_\eps\) is bounded and supported on an \(\eps\) collar.
	
	Denoting \(\Theta_\eps\) the total mass of \(q_\eps\) and \(\Xi_\eps\) its second raw moment,
	we prove that, if
	\(\Theta_\varepsilon\to\infty\) and \(\Xi_\varepsilon\to0,	\)
	then the glued processes converge, from every deterministic sequence of starting points, to Brownian motion with weigthed Kirchhoff interface conditions.
	No reversibility of $Q$, pointwise scaling ansatz for $q_\varepsilon$, or first-moment condition is required. Furthermore, we show that the second fundamental form of the pages contributes no additional interface term: curvature remains only through the bulk Laplace--Beltrami operators.
\end{abstract}




\end{frontmatter}
\section{Introduction}
\label{sec:introduction}

Diffusion on a space with branching or interfaces is not determined solely by its
behavior away from the singular set.  One must also specify how the process
crosses, reflects from, or redistributes itself when it reaches the singularity.
On a collection of rays meeting at a point, the classical example is Walsh
Brownian motion, whose radial component is reflecting Brownian motion and whose
excursions are distributed among the rays according to prescribed weights
\cite{Walsh1978,BarlowPitmanYor1989}.  More generally, diffusion processes on
metric graphs are governed by second-order operators on the edges together with
conditions at the vertices, and they arise naturally in averaging problems,
small-parameter limits and diffusion through thin structures
\cite{FreidlinWentzell1993,FreidlinSheu2000,KostrykinPotthoffSchrader2012}.

The higher-dimensional singular geometry considered here is an open book.  Let
$(M_i,g_i)$, $i=1,\ldots,k$, be compact smooth $n$-dimensional Riemannian
manifolds with smooth boundary, and identify every $\partial M_i$ with the same
compact $(n-1)$-dimensional manifold $\Gamma$.  We assume that the metrics
induced on the binding agree
and form the quotient
\[
\cM
=
\left(\bigsqcup_{i=1}^k M_i\right)\Big/\sim,
\qquad
x\in\partial M_i\sim x\in\partial M_j.
\]
Thus $\cM$ is smooth away from $\Gamma$ and branched along the common binding.
The pages carry the
same tangential metric at $\Gamma$, but their first normal jets may differ.
In particular, the pages in
our model may have different second fundamental forms, different mean
curvatures and different interior geometries.

Diffusions on branched spaces with gluing conditions at the singular set have a
substantial history.  Freidlin and Wentzell introduced open-book diffusions in
the context of averaging and formulated the limiting process through pagewise
differential operators together with gluing conditions at the binding
\cite{FreidlinWentzell2004,FreidlinWentzell2012}.  Their open-book
model in \cite{FreidlinWentzell2004} has a product structure inherited from a
graph-type slow variable, while the discussion in
\cite[Chapter~9]{FreidlinWentzell2012} points out the additional analytic
problem created by a general open book: one needs sufficiently regular solutions
of elliptic or parabolic equations on the singular space in order to identify a
Markov process by its generator. More recently, Akduman and Kuchment
\cite{AkdumanKuchment2024} developed an operator-theoretic framework for
Laplace-type operators on compact Riemannian open books with smooth pages and
general junction conditions. Related intrinsic constructions have been
developed for Euclidean and Riemannian complexes
\cite{BrinKifer2001,Bouziane2005}, for simple stratified spaces
\cite{NyeWhite2014}, for Riemannian polyhedral complexes
\cite{PivarskiSaloffCoste2008}, and for strip complexes
\cite{BendikovSaloffCosteSalvatoriWoess2011}.

Our question is instead a selection problem: \emph{which interface law is
	produced by a microscopic dynamics concentrated near the binding?}
This places the paper next to several singular-limit mechanisms in probability.
A first class of results obtains singular diffusions through a degeneration of
the underlying state space.  This includes the narrow-tube example of
Freidlin and Wentzell \cite{FreidlinWentzell1993}, as well as subsequent
limits of diffusions or Laplacians on shrinking tubular and graph-like
domains \cite{AlbeverioKusuoka2012,ExnerPost2005}.  In higher dimension, Corbin
\cite{Corbin2020} proves spectral convergence from uniformly fattened
open-book structures to an operator with Kirchhoff coupling on the limiting
open book.

A different mechanism produces a singular state space by eliminating fast
degrees of freedom.  Freidlin and Wentzell
\cite{FreidlinWentzell1993} already treat examples of this type for processes
with a fast discrete component and for randomly perturbed Hamiltonian
systems, where the limiting graph is a quotient by the fast motion.
Their higher-dimensional open-book result
\cite{FreidlinWentzell2004} belongs to this averaging paradigm: the open book
is obtained by identifying connected components of common level sets of the
first integrals, and the projected slow motion converges weakly to a diffusion
on that quotient.


A separate line of work concerns homogenization across interfaces for
continuous diffusions.  Hairer and Manson
\cite{HairerManson2010,HairerManson2011} study diffusions whose coefficients
are periodic away from an interface region and interpolate between the two
media near that interface.  Under diffusive rescaling, the bulk dynamics
homogenize while the shrinking transition region leaves a singular
contribution to the limiting process, expressed through skewness or an
interface local-time drift.

The mechanism considered here is different.  The continuous motion on each
page is Brownian and a finite-state page
label jumps at bounded rates only inside a shrinking collar of the
binding.  The singular limit is therefore driven by localized fast switching; the invariant law of
the switching chain survives in the weighted Kirchhoff condition.


For every $\eps>0$ our microscopic state space contains no glued interface.  We
first modify $g_i$ only in an auxiliary collar of width $\rho_\eps$, where
$\rho_\eps=o(\eps)$, so that the metric is exactly product near $\Gamma$.  This
allows us to form a smooth boundaryless double $(\widetilde M_i,\widetilde
g_i^\eps)$ of each page.  On the disjoint union
\[
\widetilde E=\bigsqcup_{i=1}^k\widetilde M_i
\]
we run ordinary Brownian motion on the current double and change the page label
only inside the collar $|s|\le\eps$, at the rates
\(
\lambda_{ij}^\eps(y,s)
=q_\eps(|s|)Q_{ij},
\ i\neq j.
\)
Here $Q$ is an arbitrary irreducible generator on $\{1,\ldots,k\}$ and
$\pi=(\pi_1,\ldots,\pi_k)$ is its invariant probability vector.  No detailed
balance assumption is imposed.  For each fixed $\eps$, the function $q_\eps$ is
bounded, so the microscopic process can be constructed directly by Poisson
thinning.  It is conservative, càdlàg, Feller and strong Markov, and only
finitely many page switches occur on compact time intervals.  Folding the
doubles back onto the original pages produces reflected Brownian motion between
switching times.

The candidate limit is the weighted Kirchhoff diffusion on $\cM$.  Its generator
acts in the interior of page $i$ as
\[
Af=\frac12\Delta_{g_i}f_i,
\]
and its smooth domain is characterized at the binding by
\(
f_1|_\Gamma=\cdots=f_k|_\Gamma\) and \(
\sum_{i=1}^k\pi_i\,\partial_{\nu_i}f_i=0.
\)
Although related Kirchhoff diffusions are known in several singular settings,
we construct this target process in the precise Riemannian geometry needed for
the convergence argument.  There are two reasons for doing so.  First, our perturbed-test-function argument requires a smooth core of the
limiting generator on which the normal derivatives at the binding, the
Kirchhoff condition, and the pagewise Laplacians are all defined.
This is needed to construct the boundary-layer correctors and to compare the
microscopic generators with the limiting one pointwise.
Second, identification of subsequential limits requires a well-posed martingale
problem for exactly that core.

Only two integrated quantities from the switching profile enter the limiting
regime:
\[
\Theta_\eps
:=\int_0^\eps q_\eps(r)\,dr,
\qquad
\Xi_\eps
:=\int_0^\eps r^2q_\eps(r)\,dr.
\]
Our assumptions are
\[
\Theta_\eps\longrightarrow\infty,
\qquad
\Xi_\eps\longrightarrow0.
\]
The first quantity is the total switching mass seen on the normal local-time
scale.  Its divergence forces the rapidly switching page label to equilibrate on the interface scale, thereby forcing equilibration according to the invariant law $\pi$. The second quantity measures the
second-order spatial effect of switching at distance $r$ from the binding. 

The introduction of a boundary-layer corrector allows us to bridge the mismatch between the microscopic Neumann domain and the
limiting Kirchhoff domain, to offset the effect of the first moment \(\int_0^\eps r q_\eps(r)\,dr\) in the limit, and thus to not require a separate assumption, and to show that the extinsic curvature is 
asymptotically negligible in the interface law.

Our main theorem states that, if
$\Theta_\eps\to\infty$ and $\Xi_\eps\to0$, then the glued microscopic
observations converge in path space, from every deterministic sequence of
starting points whose images converge in $\cM$, to the weighted Kirchhoff
diffusion.  This includes starting points approaching the binding and arbitrary
microscopic page labels; if the limiting initial point lies on $\Gamma$, the
limiting law is independent of those labels.  In this sense both continuity
through the binding and the weighted flux balance are emergent: neither appears
in the microscopic state space or in the microscopic generator.

The contributions can therefore be summarized as follows.  We derive a
weighted Kirchhoff law from bounded-rate, boundaryless microscopic
jump-diffusions on arbitrary smooth Riemannian pages; the page chain need only be
irreducible, not reversible; the scaling regime is expressed by the two moments
$\Theta_\eps$ and $\Xi_\eps$, with no first-moment condition; convergence holds
from deterministic initial conditions, including at the binding; and the
resulting interface law is universal under changes of the second fundamental
forms of the pages that preserve the common induced metric on $\Gamma$.

The rest of the paper is organized as follows.  Section~\ref{sec:geometry}
introduces the open-book geometry and Fermi collars.  Section \ref{sec:target} construct the
target Kirchhoff diffusion and its smooth core. Section \ref{sec:approximants} followed by the boundaryless
microscopic processes and their folded observations.  Section \ref{sec:tightness} is concerned with tightness, while in \Cref{sec:corrector}
we introduce the non-symmetric boundary-layer corrector Finally \Cref{sec:main-limit} proves the final convergence theorem and records
the curvature-universality consequence.

The baseline model treated here changes only the page label:
\(
(i,y,r)\mapsto(j,y,r).
\)
A natural next question is to replace this by microscopic landing maps
\[
(i,y,r)
\longmapsto
\left(
j,\exp_y^\Gamma\Delta_{ij}^{\parallel,\eps},
r+\Delta_{ij}^{\perp,\eps}
\right).
\]
One can expect additional surface drift, surface diffusion, skew flux or Wentzell-type term in the landing-map theory can therefore to be generated by this additional the landing mechanism.

\subsection{Notation}
\begin{itemize}
	\item For a manifold with boundary \(M\), the notation
	\(C^m(M)\) means \(C^m\) up to the boundary.
	\item The space \(D([0,T],M)\) denotes the space of càdlàg \(M\)-valued
	paths endowed with the Skorokhod \(J_1\) topology, and
	\(C([0,T],M)\) denotes its subspace of continuous paths. A family of
	probability measures on \(D([0,T],M)\) is said to be
	\emph{\(C\)-tight} if it is tight and every weak subsequential limit
	is concentrated on \(C([0,T],M)\). We write
	\(\Longrightarrow\) for weak convergence of probability measures on
	the path space under consideration.
	\item For a process \(X\) started from \(x\), we write
	\(\mathbb P_x\) and \(\mathbb E_x\) for its law and expectation,
	respectively, and \(\mathcal L_x(X)\) for its law as a path-valued
	random variable.
	\item The letter \(C\) denotes a finite positive constant whose value may
	change from line to line. Unless explicitly indicated otherwise,
	such constants are independent of \(\varepsilon\), of the page index
	\(i\), and of the initial point. A subscript, as in \(C_T\), indicates
	permitted dependence on the corresponding parameter.
\end{itemize}

\section{The one-interface geometry}
\label{sec:geometry}

\subsection{Pages and the glued space}

Fix integers $n\ge2$ and $k\ge2$.  For $i=1,\dots,k$, let $(M_i,g_i)$ be a compact connected smooth $n$-dimensional Riemannian manifold with smooth boundary.  Assume that every boundary is identified with the same compact connected smooth $(n-1)$-manifold $\Gamma$ through fixed diffeomorphisms, and use these identifications to write
\(
\partial M_i=\Gamma.
\)
We assume that the induced boundary metrics coincide, that is
\(
g_i|_{T\Gamma}=h\), \( i=1,\dots,k,
\)
for one Riemannian metric $h$ on $\Gamma$.


The glued open-book space is
\[
\cM
:=
\left(\bigsqcup_{i=1}^k M_i\right)\Big/\sim,
\qquad
x\in\partial M_i\sim x\in\partial M_j.
\]
We denote the common image of the boundaries again by $\Gamma$.  The quotient is equipped with the intrinsic path metric obtained from the page metrics.  In particular, if two points lie at the same Fermi footpoint $y\in\Gamma$ and at normal distances $r,s$ on distinct pages, then the path through the binding gives
\(
d_{\cM}\big((i,y,r),(j,y,s)\big)\le r+s.
\)

We shall also use the separated page space
\[
E^\sharp:=\bigsqcup_{i=1}^k M_i
\]
and the quotient map
\[
\mathfrak q:E^\sharp\longrightarrow\cM.
\]
Functions on $\cM$ are identified with tuples $f=(f_1,\dots,f_k)$ whose traces agree on $\Gamma$.

\subsection{Fermi collars}

By the tubular-neighborhood theorem, for each page \(M_i\) there exists
\(r_i>0\) such that the inward normal exponential map
\[
\Phi_i:\Gamma\times[0,r_i)\longrightarrow M_i,
\qquad
\Phi_i(y,r)=\exp_y^{M_i}\!\bigl(r\nu_i(y)\bigr),
\]
is a diffeomorphism onto a neighborhood of \(\Gamma\) in \(M_i\), where
\(\nu_i\) denotes the inward unit normal along \(\Gamma\).
Since \(\Gamma\) is compact and the number of incident pages is finite, after
decreasing the radii if necessary we may choose a common \(r_0>0\).
We refer to the image of
\(\Phi_i:\Gamma\times[0,r_0)\to M_i\) as a \emph{Fermi collar} of
\(\Gamma\) in \(M_i\), and to \((y,r)\) as the corresponding Fermi
coordinates.

By the Gauss lemma, the metric in these coordinates has the form
\begin{equation}
	\Phi_i^\ast g_i
	=dr^2+h_i(y,r),
	\qquad
	h_i(y,0)=h(y),
	\label{eq:fermi-metric}
\end{equation}
where \(h_i(\cdot,r)\) is the metric induced on the hypersurface at
normal distance \(r\) from \(\Gamma\), and \(h\) is the common metric
induced on \(\Gamma\). In the following, with a slight abuse of notation, we will denote the pullback metric \(\Phi_i^\ast\) by \(g_i\).
In particular, its first normal derivative at \(r=0\) is, up to the
chosen sign convention, twice the second fundamental form of
\(\Gamma\subset M_i\).

For a symmetric \((0,2)\)-tensor \(T\), we write, in Einstein notation, \(|T|_h^2:=h^{ac}h^{bd}T_{ab}T_{cd}\) and \(\norm T_{C^0(\Gamma,h)}:=\sup_{y \in \Gamma}|T(y)|_h\). Smoothness, compactness and the finite number of pages \(M_i\) imply the uniform expansion
\begin{equation}
	\norm{h_i(\cdot,r)-h}_{C^0(\Gamma)}\le Cr, \quad
	\sup_{y\in\Gamma,\;0\le r\le r_0}
	|\partial_rh_i(y,r)|_h
	\le C .
	\qquad 0\le r\le r_0,
	\label{eq:metric-first-order}
\end{equation}
with a constant independent of $i$.

Denoting by \(d\sigma_h\) the volume element of the boundary metric \(h\), we may write the Riemannian volume element in the collar as
\begin{equation}
	d\vol_{g_i}=J_i(y,r)\,dr\,d\sigma_h(y),
	\qquad
	J_i(y,r):=\left(\frac{\det h_i(y,r)}{\det h(y)}\right)^{1/2}.
	\label{eq:Jacobian}
\end{equation}
Then $J_i(y,0)=1$, and smoothness gives
\begin{equation}
	J_i(y,r)=1+\kappa_i(y)r+O(r^2),
	\qquad
	\kappa_i(y):=\partial_r\log J_i(y,0).
	\label{eq:Jacobian-expansion}
\end{equation}
Up to the sign convention for the chosen inward normal, $\kappa_i$ is the trace of the second fundamental form of $\Gamma$ as the boundary of page $i$.  Thus the second fundamental form enters the Fermi metric at first order in $r$, while its trace enters the volume density through \eqref{eq:Jacobian-expansion}.

\section{The target Kirchhoff diffusion}
\label{sec:target}

Diffusions on branched spaces governed by differential operators on the
individual pages and gluing conditions at the binding are classical, as discussed in the Introduction.

The most closely related analytic constructions include the following. For Riemannian
polyhedral complexes, Pivarski and Saloff-Coste
\cite[Proposition~1.21 and Theorem~3.10]{PivarskiSaloffCoste2008}
identify the Kirchhoff condition on codimension-one faces and, under
their geometric assumptions, obtain a Feller heat semigroup on
\(C_0\).  Weighted Dirichlet forms, Kirchhoff-type conditions, and
regularity of weak solutions are developed for strip complexes in
\cite{BendikovSaloffCosteSalvatoriWoess2011}.  More recently, Akduman
and Kuchment \cite{AkdumanKuchment2024} developed an
operator-theoretic framework for compact Riemannian open books with
smooth pages, treating pagewise Laplace--Beltrami operators subject to
general local junction conditions at the bindings.

These works provide closely related realizations of differential
operators on branched spaces, but the convergence argument below
requires a particular \(C(\cM)\)-realization of the weighted Kirchhoff
diffusion.  In particular, we need a smooth core on which the normal
derivatives, the pagewise Laplacians, and the transmission condition are
defined classically, together with a well-posed martingale problem from
every starting point.  We therefore record the precise realization
associated with the weights \(\pi_1,\ldots,\pi_k\).

Define the finite measure \(\mu\) on \(\cM\) by
\begin{equation*}
	\int_{\cM} f\,d\mu
	:=
	\sum_{i=1}^k
	\pi_i\int_{M_i}f_i\,d\vol_{g_i}.
	\label{eq:target-measure}
\end{equation*}
There is no ambiguity at the binding, which has zero pagewise volume.
Let
\begin{align}
	\cF_K
	&:=
	\left\{
	f=(f_i):
	f_i\in H^1(M_i),\
	\Tr f_1=\cdots=\Tr f_k\text{ on }\Gamma
	\right\},
	\notag\\
	\cE_K(f,\varphi)
	&:=
	\frac12\sum_{i=1}^k
	\pi_i\int_{M_i}
	\ip{\nabla^{g_i}f_i}{\nabla^{g_i}\varphi_i}_{g_i}
	\,d\vol_{g_i}.
	\label{eq:EK}
\end{align}
The trace maps \(H^1(M_i)\to H^{1/2}(\Gamma)\) are continuous, so
\(\cF_K\) is a closed subspace of
\(\bigoplus_i H^1(M_i)\).  Since the weights \(\pi_i\) are strictly
positive, \(\cE_K+\langle\cdot,\cdot\rangle_{L^2(\mu)}\) defines an
equivalent Hilbert norm on \(\cF_K\).  Moreover, normal contractions
preserve the common-trace condition and decrease the energy; thus
\((\cE_K,\cF_K)\) is the natural weighted Dirichlet form associated
with the open book.

Let \(\nu_i\) denote the inward unit normal to \(\Gamma\) in \(M_i\).
For smooth \(u\), pagewise Green's formula shows the origin of the
weights:
\begin{align*}
	\cE_K(u,\varphi)
	&=
	-\frac12\sum_{i=1}^k\pi_i
	\int_{M_i}(\Delta_{g_i}u_i)\varphi_i\,d\vol_{g_i}
	-\frac12\int_\Gamma
	\left(
	\sum_{i=1}^k\pi_i\partial_{\nu_i}u_i
	\right)
	\Tr\varphi\,d\sigma_h ,
	\label{eq:green-kirchhoff}
\end{align*}
where the minus sign in the boundary term is due to our convention that
\(\nu_i\) points inward.  Thus the natural flux condition associated
with \eqref{eq:EK} is
\(\sum_{i=1}^k\pi_i\partial_{\nu_i}u_i=0\) on \(\Gamma.\)

Let
\begin{equation*}
	\mathcal G
	:=
	\left\{
	G\in C(\cM):
	G_i\in C^\infty(M_i)\text{ for every }i
	\right\}.
	\label{eq:Gclass}
\end{equation*}
For the \(C(\cM)\)-realization, define
\begin{equation}
	\cC_K
	:=
	\left\{
	\begin{array}{l}
		f\in \mathcal G, \
		\displaystyle
		\sum_{i=1}^k\pi_i\partial_{\nu_i}f_i=0
		\text{ on }\Gamma, \
		\displaystyle
		\Delta_{g_1}f_1|_\Gamma
		=
		\cdots
		=
		\Delta_{g_k}f_k|_\Gamma
	\end{array}
	\right\}.
	\label{eq:CK}
\end{equation}
The two conditions at the binding have different roles.  The first is
the transmission condition.  The second is not an additional
transmission law: it guarantees that the pagewise expression
\(\frac12\Delta_{g_i}f_i\) extends to a continuous function on the
quotient.  This is the analogue, in the present setting, of the common
limiting value of the pagewise generator required in the domains used
by Freidlin and Wentzell~\cite[Chapter~9, Section~3]
{FreidlinWentzell2012}.

For \(f\in\cC_K\), set
\begin{equation*}
	(A_0f)|_{M_i^\circ}
	:=
	\frac12\Delta_{g_i}f_i
	\label{eq:A0}
\end{equation*}
and define \(A_0f\) on \(\Gamma\) by the common boundary value in
\eqref{eq:CK}. 

\begin{lemma}
	\label{lem:target-density-pmp}
	The space $\cC_K$ is uniformly dense in $C(\cM)$.  Moreover $A_0$ satisfies the positive maximum principle and $A_0\bone=0$.
\end{lemma}

\begin{proof}
	We first prove density.  Let $f\in C(\cM)$ and $\eta>0$.  Choose a smooth
	$\phi$ on $\Gamma$ with
	\(
	\norm{\phi-f|_\Gamma}_\infty<\eta.
	\)
	Fix a smooth collar cutoff $\chi$ such that \(\chi(r)=1\) for \(0\le r \le r_0/4\), and \(\chi(r)=0\) for \(r \ge r_0/2\). Extending $\phi$
	to every page by $\chi(r)\phi(y)$ and defining \(g_i(y,r):=f_i(y,r)-\chi(r)\phi(y)\) we obtain \(\sup_i \|g_i(y,0)\|_{L^\infty(\Gamma)}\le \eta\). 
	For \(\delta >0\), select a smooth cutoff \(\theta_\delta\) such that \(\theta_\delta(r)=0\) for \(0\le r \le \delta\), and \(\theta_\delta=1\) for \(r \ge 2\delta\). Defining \(g_i^\delta(y,r):=\theta_\delta(r)g_i(y,r)\), we have
	\[\|g_i^\delta-g_i\|_\infty\le \sup_{\substack{y\in\Gamma\\0\le r\le 2\delta}}|g_i(y,r)|\le \sup_{\substack{1\le i\le k, y\in\Gamma\\0\le r\le 2\delta}}|g_i(y,r)-g_i(y,0)|+\eta.\]
	Let us denote the supremum in the right hand side by \(\omega(\delta)\). Because the \(g_i\)'s are continuous up to the boundary, \(\Gamma\) is compact and there are finitely many pages we have \(\omega(\delta)\to 0\) as \(\delta \to 0\). Therefore \(\|g_i^\delta-g_i\|_\infty \le o_\delta(1)+\eta\). 
	
	Now, we can choose smooth approximations \(q_i \in C^\infty(M_i)\) supported outside of the collar \(\{(y,r):0\le r \le \delta/2\}\) such that \(\|q_i-g_i^\delta\|_\infty >\eps\) for all \(\eps >0\). Define \(u_i(y,r):=\chi(r)\phi(y)+q_i(y,r)\). Each \(u_i\) is smooth and, since \(u_1|_\Gamma=\cdots=u_k|_\Gamma=\phi\), they glue into a continuous \(u \in C(\mathcal M)\). Moreover, by the triangle inequality and the finite number of pages
	\(\|u-f\|_{C(\cM)}\le \eps+o_\delta(1)+\eta\) which proves the density of \(\mathcal G\).

	Write \(
	S(y):=\sum_{i=1}^k\pi_i \partial_{v_i}u_i(y,0).
	\)
	Choose a smooth cutoff $\chi_\gamma$ supported in $0\le r<2\gamma$ and equal to one for
	$0\le r\le\gamma$.  The functions
	\(
	u_i^{(1)}(y,r)
	:=
	u_i(y,r)-r\chi_\gamma(r)S(y)
	\)
	have the same trace as $u$ and satisfy
	\(
	\sum_i\pi_i\partial_{\nu_i}u_i^{(1)}=0\text{ on }\Gamma.
	\)
	Moreover, we have $\max_i\|u_i^{(1)}-u_i\|_\infty \le 2\gamma\|S\|_\infty$.
	
	Set
	\(
	d_i(y):=\Delta_{g_i}u_i^{(1)}(y,0), \
	\bar d(y):=\sum_i\pi_i d_i(y),
	\) and \(D^i(y,r):=r^2\chi_\gamma(r)(\bar d(y)-d_i(y))\). Since all derivaties of \(\chi_\gamma\) vanish at \(r=0\) the \(d_i\) remain bounded as \(\gamma \to 0\).
	Define
	\(
	u_i^{(2)}(y,r)
	:=
	u_i^{(1)}(y,r)
	+
	1/2 \, D^i(y,r).
	\)
	Because $\chi_\gamma\equiv1$ near $r=0$, we have
	\(
	\partial_rD^i|_{r=0}=0,
	\)
	while, using
	\begin{equation}
		\label{eq:fermi-Delta}
		\Delta_{g_i}=\partial_{rr}+\kappa_i\partial_r+\Delta_{h_i(r)},
	\end{equation} we obtain
	\(
	\Delta_{g_i}D^i|_{r=0}
	=\bar d-d_i
	\) and consequently \(\Delta_{g_i}u_i^{(2)}|_\Gamma=\Delta_{g_i}u_i^{(1)}|_\Gamma+ \Delta_{g_i}D^i|_\Gamma=\bar d.\)
	Thus this second correction leaves the common trace and the weighted Kirchhoff
	condition unchanged and gives
	\(
	\Delta_{g_i}u_i^{(2)}|_\Gamma=\bar d
	\quad\text{for all }i.
	\)
	Moroever, \(\max_i\|u_i^{(2)}-u_i^{(1)}\|_\infty \le 2\gamma^2\max_i\|\bar d-d_i\|_\infty=O(\gamma^2)\)  Taking $\gamma$ sufficiently small proves that $\cC_K$ is uniformly dense.
	
	For the positive maximum principle, suppose $f\in\cC_K$ attains its global maximum at $z$.  If $z$ lies in a page interior, the usual elliptic maximum principle gives $A_0f(z)\le0$.  If $z=(y,0)\in\Gamma$, the maximum principle along \(\Gamma\) gives, in Fermi coordinates, \(\Delta_{h_i(0)} f_i(y,0)=\Delta_hf_\Gamma(y)\le 0.\) In the normal directions, \(\partial_r f_i(z)\le 0\).
	The weighted Kirchhoff condition and $\pi_i>0$ then imply
	$\partial_{r}f_i(z)=0$ for every $i$. A taylor expansion in \(r\) together with \(\partial_r f_i(z)\le 0\) gives \(\partial_{rr}f_i(z) \ge 0.\)
	Therefore, using \eqref{eq:fermi-Delta}, we obtain
	$\Delta_{g_i}f_i(z)\le0$.  Thus $A_0f(z)\le0$ also at the binding.  Finally, $A_0\bone=0$ is immediate.
\end{proof}

\begin{lemma}
	\label{lem:weak-smooth}
	Let $G\in\mathcal G$, and let $u\in\cF_K$ solve
	\begin{equation}
		\alpha\ip{u}{\varphi}_{L^2(\mu)}+\cE_K(u,\varphi)
		=\ip{G}{\varphi}_{L^2(\mu)},
		\qquad \varphi\in\cF_K.
		\label{eq:target-variational}
	\end{equation}
	Then $u_i\in C^\infty(M_i)$ for every page.  Moreover
	\begin{equation}
		\sum_{i=1}^k\pi_i\partial_{\nu_i}u_i=0
		\quad\text{on }\Gamma.
		\label{eq:weak-kirchhoff}
	\end{equation}
\end{lemma}

\begin{proof}
	For smooth functions compactly supported in \(M_i^\circ\), \eqref{eq:target-variational} reduces on page \(i\) to the weak uniformly elliptic equation 
	\begin{equation}
		\label{eq:interior-pde}
		\alpha u_i -\frac12 \, \Delta_{g_i} u_i =G_i.
	\end{equation}
	Standard elliptic interior regularity and bootstrapping imply \(u_i \in C^\infty(M_i^\circ)\).  We prove regularity up to the binding directly.  Fix
	$U_0\Subset U_1\Subset\Gamma$ inside one coordinate chart and
	$0<r_1<r_0$. Denote \(Q_0:=U_0\times(0,r_1/2)\) and \(Q_1:=U_1\times(0,r_1)\). Use Fermi coordinates $(y,r)$ on
	$U_1\times[0,r_1)$.  Write
	\(
	d\vol_{g_i}=K_i(y,r)\,dy\,dr,
	\), with \(K_i(y,r):=J_i(y,r)\sqrt{\det h(y)}\),
	and set, with Greek indices ranging over the \((y,r)\) coordinates,
	\(
	B_i^{\alpha\beta}
	:=K_i(g_i^{-1})^{\alpha\beta}.
	\)
	The matrices $B_i$ are smooth, symmetric and uniformly positive definite on the
	fixed half-cylinder.  In particular, for some $0<\lambda\le\Lambda<\infty$,
	\begin{equation}
		\lambda|\xi|^2
		\le
		B_i^{\alpha\beta}(y,r)\xi_\alpha\xi_\beta
		\le
		\Lambda|\xi|^2
		\label{eq:local-uniform-ellipticity}
	\end{equation}
	uniformly in $i$.

	We first treat tangential derivatives. Choose a cutoff
	$\eta\in C_c^\infty(U_1\times[0,r_1))$ which equals one on
	$U_0\times[0,r_1/2]$; the cutoff is the same on every page.  Let $e_a$ be a
	tangential coordinate direction and $h_0>0$ be small enough so that \(\supp \tau_h^a \eta \subset U_1 \times [0,r_1)\) for \(|h|<h_0\). Set \(\tau_h^a \eta(y,r)=\eta(y+he_a,r)\) and
	\[
	w_h:=D_h^a u:=\frac{\tau_h^a u(y,r)-u(y,r)}{h}.
	\]
	Since the traces $\Tr u_i$ agree, the traces $\Tr w_{h,i}$ also agree.  Hence
	\(
	\varphi_i:=-D_{-h}^a(\eta^2w_{h,i})
	\)
	has a common trace and is an admissible element of $\cF_K$.
	
	For compactly supported functions in the $y$-chart the discrete integration by
	parts identity is
	\begin{equation}
		\int F\,D_{-h}^a H
		=-\int (D_h^aF)\,H.
		\label{eq:discrete-IBP}
	\end{equation}
	The discrete Leibniz rule gives
	\begin{align}
		D_h^a(B_i^{\alpha\beta}\partial_\beta u_i)
		&=
		(\tau_hB_i^{\alpha\beta})\partial_\beta w_{h,i}
		+(D_h^aB_i^{\alpha\beta})\partial_\beta u_i,
		\label{eq:DQ-product-B}\\
		D_h^a(K_i u_i)
		&=
		(\tau_hK_i)w_{h,i}+(D_h^aK_i)u_i,
		\label{eq:DQ-product-J}\\
		D_h^a(K_iG_i)
		&=
		(\tau_hK_i)D_h^aG_i+(D_h^aK_i)G_i.
		\label{eq:DQ-product-G}
	\end{align}
	Substituting \(\varphi_i\) in \eqref{eq:target-variational} and applying \eqref{eq:discrete-IBP} therefore yields
	\begin{align}
		&\frac12\sum_i\pi_i\int
		(\tau_hB_i^{\alpha\beta})
		\partial_\beta w_{h,i}\,
		\partial_\alpha(\eta^2w_{h,i})
		+\alpha\sum_i\pi_i\int
		(\tau_hK_i)\eta^2|w_{h,i}|^2=
		\sum_i\pi_i\int D_h^a(K_iG_i)\,\eta^2w_{h,i}
		\notag\\
		&\qquad\quad-\alpha\sum_i\pi_i\int
		(D_h^aK_i)u_i\,\eta^2w_{h,i}
		-\frac12\sum_i\pi_i\int
		(D_h^aB_i^{\alpha\beta})\partial_\beta u_i\,
		\partial_\alpha(\eta^2w_{h,i}).
		\label{eq:DQ-identity}
	\end{align}
	All integrals in this step are over \(Q_1\). We name terms in \eqref{eq:DQ-identity} respectively \(S_1+S_2=S_3-S_4-S_5\) and estimate them.
	
	\(S_1\) equals
	\[
	\frac12\sum_i\pi_i\int
	\eta^2(\tau_hB_i^{\alpha\beta})
	\partial_\beta w_{h,i}\partial_\alpha w_{h,i}
	+\sum_i\pi_i\int
	\eta w_{h,i}(\tau_hB_i^{\alpha\beta})
	\partial_\beta w_{h,i}\partial_\alpha\eta=:S_1^{a}+S_1^{b}.
	\]
	We estimate these two terms separately.  Since the ellipticity bound
	\eqref{eq:local-uniform-ellipticity} also holds for $\tau_hB_i$, writing
	$\pi_*:=\min_i\pi_i>0$,
	\(
	S_1^a
	\ge
	c
	\sum_i\|\eta\nabla w_{h,i}\|_{L^2}^2,\) with \(c:=\frac{\lambda\pi_*}{2}.
	\) Since $B_i$ and $\nabla\eta$ are uniformly
	bounded and the number of pages is finite, Young's inequality with parameter $c/(8C)$ gives
	\begin{align}
		\left|S_1^b
		\right|	&\le
		C\sum_i
		\|\eta\nabla w_{h,i}\|_{L^2}\|w_{h,i}\|_{L^2}
		\le
		\frac{c}{8}\sum_i\|\eta\nabla w_{h,i}\|_{L^2}^2
		+C\sum_i\|w_{h,i}\|_{L^2}^2.
		\label{eq:DQ-cutoff-error}
	\end{align}
	
	We next estimate \(S_5\).  The uniform
	$C^1$ bound on $B_i$ implies, by the fundamental theorem of calculus,
	\[
	D_h^aB_i^{\alpha\beta}(y,r)
	=
	\int_0^1
	\partial_{y_a}B_i^{\alpha\beta}(y+the_a,r)\,dt,
	\]
	and hence
	\(
	\sup_{|h|<h_0}
	\|D_h^aB_i\|_{L^\infty}
	\le C
	\)
	uniformly in $i$.  Expanding
	\(
	\partial_\alpha(\eta^2w_{h,i})
	=
	\eta^2\partial_\alpha w_{h,i}
	+2\eta(\partial_\alpha\eta)w_{h,i}
	\), using $0\le\eta\le1$ and the boundedness of $\nabla\eta$, and applying Young inequality twice,
	we obtain
	\begin{align}
		|S_5|&\le
		C\sum_i
		\|\nabla u_i\|_{L^2}
		\|\eta\nabla w_{h,i}\|_{L^2}
		+
		C\sum_i
		\|\nabla u_i\|_{L^2}\|w_{h,i}\|_{L^2}
		\notag \\&\le
		\frac{c}{8}\sum_i\|\eta\nabla w_{h,i}\|_{L^2}^2
		+
		C\sum_i\left(
		\|u_i\|_{H^1}^2+\|w_{h,i}\|_{L^2}^2
		\right).
		\label{eq:DQ-B-error}
	\end{align}
	
	Similarly, the uniform $C^1$ bound on $K_i$ gives
	\(
	\sup_{|h|<h_0}\|D_h^aK_i\|_{L^\infty}\le C.
	\)
	Consequently,
	\begin{align}
		\left|
		S_4
		\right|
		&\le
		C\sum_i\|u_i\|_{L^2}\|w_{h,i}\|_{L^2}
		\le
		C\sum_i\left(
		\|u_i\|_{L^2}^2+\|w_{h,i}\|_{L^2}^2
		\right).
		\label{eq:DQ-J-error}
	\end{align}

	
	Since $J_iG_i\in H^1$ on \(U_1 \times [0,r_1)\) and\(\tau_h^a\) respects this domain, we have
	\(
	\|D_h^a(J_iG_i)\|_{L^2}
	\le
	\|\partial_{y_a}(J_iG_i)\|_{L^2}.
	\)
	Using the product rule and the uniform $C^1$ bounds on $J_i$,
	\(
	\|\partial_{y_a}(J_iG_i)\|_{L^2}
	\le
	C(
	\|G_i\|_{L^2}
	+\|\partial_{y_a}G_i\|_{L^2})
	\le
	C\|G_i\|_{H^1}.
	\)
	Hence
	\(
	\sup_{|h|<h_0}
	\|D_h^a(J_iG_i)\|_{L^2}
	\le
	C\|G_i\|_{H^1}.
	\)
	Therefore, by Cauchy--Schwarz, $0\le\eta\le1$ and Young's inequality
	\begin{align}
		\left|S_3
		\right|
		&\le
		C\sum_i
		\|D_h^a(J_iG_i)\|_{L^2}
		\|w_{h,i}\|_{L^2}\le
		C\sum_i
		\|G_i\|_{H^1}
		\|w_{h,i}\|_{L^2}\le C\sum_i\left(
		\|G_i\|_{H^1}^2
		+\|w_{h,i}\|_{L^2}^2
		\right).
		\label{eq:DQ-forcing-error}
	\end{align}
	Finally, \(S_2\)
	is nonnegative and may therefore be discarded when deriving a lower bound.
	Combining all estimates yields
	\begin{equation}
		\sum_i\|\nabla D_h^a u_i\|_{L^2(Q_0)}^2 \le \sum_i\|\eta\nabla D_h^a u_i\|_{L^2}^2
		\le
		C\sum_i\left(
		\|u_i\|_{H^1(Q_1)}^2
		+\|G_i\|_{H^1(Q_1)}^2
		\right)<\infty,
		\label{eq:tangential-DQ-estimate}
	\end{equation}
	uniformly for small $h$. Moreover, since \(u_i\in H^1\), \(\|D_h^a u_i\|_{L^2(Q_0)}^2\le \|u_i\|_{H^1(Q_1)}<\infty\). Weak compactness of difference quotients then gives
	\(
	\partial_{y_a}u_i\in H^1(Q_0).
	\)
	Since $a$ was arbitrary, all second weak derivatives with at least one tangential
	index belong to $L^2$ up to the binding.
	
	For normal second derivatives, we again test \eqref{eq:target-variational} with functions compactly supported in the
	interior of page $i$ gives.
	Inserting the Fermi coordinates expansion
	\(
	\Delta_{g_i}
	=
	\partial_{rr}
	+\kappa_i(y,r)\partial_r
	+h_i^{ab}(y,r)\partial_{y_ay_b}
	+d_i^a(y,r)\partial_{y_a},
	\)
	with smooth bounded coefficients on the closure \(\overline Q_1\), in \eqref{eq:interior-pde}, gives
	\begin{equation}
		\partial_{rr}u_i
		=
		2\alpha u_i-2G_i-\kappa_i\partial_ru_i
		-h_i^{ab}\partial_{y_ay_b}u_i-d_i^a\partial_{y_a}u_i,
		\label{eq:recover-rr}
	\end{equation}
	where \(\kappa_i:=\partial_r \log K_i\) and \(d_i^a:=1/K_i\,\partial_{y_a}(K_ih_i^{ab})\).
	Since \(\partial_{y_a}u_i\in H^1(Q_0)\), every term on the right hand side in particular \(h_i^{ab}\partial_{y_ay_b}u_i\), is in $L^2(Q_0)$.  Hence
	$\partial_{rr}u_i\in L^2$ there and
	\(
	u_i\in H^2(Q_0).
	\)
	
	We now bootstrap the preceding $H^2$ regularity.  The main point compared with the standard higher-order difference-quotient bootstrap, is that, when tangential difference quotients are iterated,
	the corresponding test vectors must continue to satisfy the common-trace
	constraint defining $\mathcal F_K$. Fix \(m\ge 1\) and
	let $a_1,\ldots,a_m$ be tangential coordinate directions.  For sufficiently
	small parameters $\mathbf h=(h_1,\ldots,h_m)$ set
	\(
	\Delta_{\mathbf h}
	:=
	D_{h_m}^{a_m}\cdots D_{h_1}^{a_1}\) and \(
	w_{\mathbf h,i}:=\Delta_{\mathbf h}u_i.
	\)
	Let $\Delta_{\mathbf h}^{\,*}$ denote the formal adjoint of
	$\Delta_{\mathbf h}$ with respect to the coordinate measure $dy\,dr$; thus it
	is the corresponding composition of backward difference quotients, with the
	signs furnished by discrete integration by parts.  We use the test vector
	\(
	\varphi_i
	=
	\Delta_{\mathbf h}^{\,*}
	\bigl(\eta^2w_{\mathbf h,i}\bigr).
	\)
	
	Tangential
	translations commute with the trace operator, and hence
	\(
	\Tr w_{\mathbf h,i}
	=
	\Delta_{\mathbf h}(\Tr u_i)
	\)
	independently of $i$.  Since the cutoff is the same on every page, the traces
	of $\eta^2w_{\mathbf h,i}$ are again common, and applying the tangential
	operator $\Delta_{\mathbf h}^{\,*}$ preserves this property.  Thus \(\phi\) is admissible in $\mathcal F_K$.
	
	Now, the bootstrap procedure described in Appendix \ref{app:bootstrap} gives \(u_i \in C^\infty(M_i)\).	
	This makes Green's formula legitimate.  For arbitrary $\varphi\in\cF_K$ with
	smooth page components, the bulk equations cancel the volume terms in
	\eqref{eq:target-variational}, leaving
	\[
	-\frac12\int_\Gamma
	\left(\sum_i\pi_i\partial_{\nu_i}u_i\right)
	\varphi|_\Gamma\,d\sigma_h=0.
	\]
	The common boundary trace of $\varphi$ is arbitrary, so
	\eqref{eq:weak-kirchhoff} follows.  Equality of the traces was already part of
	$u\in\cF_K$.
\end{proof}

\begin{lemma}
	\label{lem:target-resolvent}
	Fix $\alpha>0$.  For every $G\in\mathcal G$ there exists a unique
	$u\in\cC_K$ such that
	\begin{equation}
		(\alpha-A_0)u=G.
		\label{eq:resolvent-equation}
	\end{equation}
	Consequently
	\(
	\mathcal G\subset(\alpha-A_0)\cC_K.
	\)
\end{lemma}

\begin{proof}
	The variational problem \eqref{eq:target-variational} has a unique
	$u\in\cF_K$ by Lax--Milgram.  By \Cref{lem:weak-smooth}, the solution is
	pagewise smooth and satisfies the weighted Kirchhoff condition.  Testing with
	functions supported in one page gives pointwise
	\(
	\alpha u_i-1/2\,\Delta_{g_i}u_i=G_i.
	\)
	Since both $u$ and $G$ have common boundary traces, taking boundary values in
	this identity yields
	\(
	\Delta_{g_i}u_i|_\Gamma
	=2\bigl(\alpha u|_\Gamma-G|_\Gamma\bigr),
	\)
	independent of $i$.  Hence $u\in\cC_K$ and
	\eqref{eq:resolvent-equation} holds.  Uniqueness follows from coercivity of the
	variational problem.
\end{proof}

\begin{proposition}
	\label{prop:target-mp}
	The closure \((A,D(A))\) of \((A_0,\cC_K)\) in
	\(C(\cM)\times C(\cM)\) generates a conservative Feller semigroup
	\(\{T_t\}_{t\ge0}\) on \(\cM\).  There is a corresponding
	c\`adl\`ag strong Markov family
	\(\{\Pp_z:z\in\cM\}\).  For every initial distribution
	\(\mu_0\in\mathcal P(\cM)\), the martingale problem for
	\((A_0,\cC_K,\mu_0)\) is well posed.
	
	In particular, \(\cC_K\) is a core for \(A\), and for
	\(f\in\cC_K\)
	\[
	(Af)|_{M_i^\circ}
	=
	\frac12\Delta_{g_i}f_i .
	\]
\end{proposition}

\begin{proof}
	Since \(\cM\) is compact metric,
	\(C(\cM)=C_0(\cM)\).  By
	\Cref{lem:target-density-pmp}, the domain \(\cC_K\) is uniformly
	dense, \(A_0\) satisfies the positive maximum principle, and
	\(A_0\bone=0\).  By \Cref{lem:target-resolvent}, for every
	\(\alpha>0\),
	\(
	\overline{\mathcal R(\alpha-A_0)}^{\|\cdot\|_\infty}
	=
	C(\cM),
	\)
	because the range contains the uniformly dense class
	\(\mathcal G\).  Hence the hypotheses of
	Ethier--Kurtz~\cite[Chapter~4, Theorem~2.2]{ethier2009markov}
	are satisfied, and the closure of \(A_0\) generates a strongly
	continuous positive contraction semigroup on \(C(\cM)\).
	Since \(\bone\in\cC_K\) and \(A_0\bone=0\), the abstract Cauchy
	problem gives
	\(
	T_t\bone=\bone,
	\)
	so the semigroup is conservative.
	
	By
	\cite[Chapter~4, Theorem~2.7]{ethier2009markov}, the Feller
	semigroup determines a c\`adl\`ag strong Markov family on \(\cM\).
	For \(f\in D(A)\), the associated process satisfies the generator
	martingale relation; see
	\cite[Chapter~4, Proposition~1.7]{ethier2009markov}.
	Consequently the Feller family solves the martingale problem for the
	restriction \((A_0,\cC_K)\).
	
	It remains to verify uniqueness for that restricted martingale
	problem.  The positive maximum principle implies dissipativity in
	the uniform norm.  Indeed, for \(\alpha>0\) and \(f\in\cC_K\),
	apply the maximum principle to \(f\) or to \(-f\) at a point where
	\(|f|\) attains \(\norm{f}_\infty\); this gives
	\begin{equation}
		\norm{(\alpha-A_0)f}_\infty
		\ge
		\alpha\norm{f}_\infty.
		\label{eq:dissipative-sup}
	\end{equation}
	Moreover,
	\(
	\overline{\cD(A_0)}^{\|\cdot\|_\infty}
	=
	\overline{\mathcal R(\alpha-A_0)}^{\|\cdot\|_\infty}
	=
	C(\cM).
	\)
	Thus the hypotheses of
	Ethier--Kurtz~\cite[Chapter~4, Theorem~4.1]{ethier2009markov}
	apply with the separating space \(L=C(\cM)\), and the martingale
	problem for \((A_0,\cC_K,\mu_0)\) is unique for every initial
	distribution \(\mu_0\).  Existence was obtained from the Feller
	family, so the martingale problem is well posed.
\end{proof}

\begin{remark}
	\label{rem:target-literature}
	We refer to the process of \Cref{prop:target-mp} as the
	\emph{weighted Kirchhoff diffusion} on the Riemannian open book.
	The proposition is included to provide the exact \(C(\cM)\)-core and
	well-posed martingale problem required below.  We do not regard the
	existence of Kirchhoff diffusions on branched spaces itself as a new
	ingredient; the distinctions from the cited constructions are the
	particular smooth Riemannian geometry, the weights
	\(\pi_i\), and the concrete generator domain used by the
	perturbed-test-function argument.
\end{remark}

\section{Boundaryless GRSD approximants}
\label{sec:approximants}

\subsection{An arbitrarily thinner smoothing collar}

Fix $\chi\in C^\infty([0,\infty);[0,1])$ such that
\(
\chi(u)=0\) on \(0\le u\le1\), and
\(
\chi(u)=1\) on \(u\ge2\).
Let $\rho_\eps$ be any positive scale satisfying
\(
0<4\rho_\eps<\eps
\) and \(\rho_\eps/\eps\to0.
\)
For instance, one may take $\rho_\eps=\eps^2$ for all sufficiently small $\eps$.
Define on the collar
\begin{equation}
	h_i^\eps(y,r)
	:=
	h(y)+\chi(r/\rho_\eps)\big(h_i(y,r)-h(y)\big),
	\qquad
	g_i^\eps:=dr^2+h_i^\eps(y,r),
	\label{eq:regularized-metric}
\end{equation}
and put $g_i^\eps=g_i$ outside $r<2\rho_\eps$.  Then
\(
g_i^\eps=dr^2+h\text{ for }0\le r\le\rho_\eps\) and
\(  g_i^\eps=g_i\text{ for }r\ge2\rho_\eps.
\)
Write
\[
d\vol_{g_i^\eps}
=J_i^\eps(y,r)\,dr\,d\sigma_h(y)\]
in Fermi coordinates.  By construction,
\(
J_i^\eps(y,r)=1\) for \(0\le r\le\rho_\eps\) and \(
J_i^\eps(y,r)=J_i(y,r)\) for \(r\ge2\rho_\eps\).

\begin{lemma}
	\label{lem:metric-comparison}
	There exist $C\ge1$ and $\eps_0>0$, independent of the page $i$,
	such that for every $0<\eps<\eps_0$ the following hold.
	
	\begin{enumerate}[label=(\roman*)]
		
		\item The regularized metrics and their duals are uniformly equivalent:
		\begin{equation}
			C^{-1}g_i\le g_i^\eps\le Cg_i,
			\qquad
			C^{-1}g_i^{-1}\le (g_i^\eps)^{-1}\le Cg_i^{-1}
			\qquad\text{on }M_i .
			\label{eq:metric-uniform-equivalence}
		\end{equation}
		
		\item In the fixed Fermi collar,
		\begin{equation}
			C^{-1}h
			\le h_i^\eps(\cdot,r)\le Ch,
			\qquad
			C^{-1}h^{-1}
			\le (h_i^\eps(\cdot,r))^{-1}\le Ch^{-1},
			\qquad 0\le r\le r_0 .
			\label{eq:tangential-metric-uniform-equivalence}
		\end{equation}
		
		\item Writing
		\(
		\kappa_i^\eps(y,r):=\partial_r\log J_i^\eps(y,r),
		\)
		one has
		\begin{equation}
			\sup_{\substack{0<\eps<\eps_0,\;1\le i\le k\\
					y\in\Gamma,\;0\le r\le r_0}}
			|\kappa_i^\eps(y,r)|
			+
			\sup_{\substack{0<\eps<\eps_0,\;1\le i\le k\\
					0\le r\le r_0}}
			\|\Delta_{h_i^\eps(r)}\|_
			{C^2(\Gamma,h)\to C^0(\Gamma)}
			\le C .
			\label{eq:fermi-uniform-coefficients}
		\end{equation}
		
	\end{enumerate}
\end{lemma}

\begin{proof}
	Because the number of pages is finite and
	$\Gamma\times[0,r_0]$ is compact, there is $C\ge1$ such that
	\(
	C^{-1}h\le h_i(\cdot,r)\le Ch
	\)
	uniformly in $i$, $y$ and $r$.  Since
	\(
	h_i^\eps
	=
	(1-\chi(r/\rho_\eps))h
	+
	\chi(r/\rho_\eps)h_i
	\)
	and $0\le\chi\le1$, \eqref{eq:tangential-metric-uniform-equivalence} follows.
	
	Outside $r<2\rho_\eps$ one has $g_i^\eps=g_i$.  Inside the Fermi
	collar, the preceding bounds imply
	\(
	C^{-1}g_i\le g_i^\eps\le Cg_i.
	\)
	Inverting the quadratic-form inequalities gives the corresponding
	dual-metric bounds, proving
	\eqref{eq:metric-uniform-equivalence}.
	
	Differentiating
	\eqref{eq:regularized-metric} in the normal direction gives
	\(
	\partial_rh_i^\eps
	=
	\rho_\eps^{-1}\chi'(r/\rho_\eps)(h_i-h)
	+
	\chi(r/\rho_\eps)\partial_rh_i .
	\)
	The second term is uniformly bounded by smoothness.  The first is
	supported where $\rho_\eps\le r\le2\rho_\eps$, and there
	$|h_i-h|_h\le Cr$; hence
	\(
	\sup_{\eps,i,y,r}
	|\partial_rh_i^\eps(y,r)|_h
	\le C.
	\)
	Since
	\(
	\kappa_i^\eps
	=
	\partial_r\log J_i^\eps
	=
	\frac12
	\Tr\!\left(
	(h_i^\eps)^{-1}\partial_rh_i^\eps
	\right),
	\)
	\eqref{eq:tangential-metric-uniform-equivalence} yields
	\(
	\sup_{\eps,i,y,r}|\kappa_i^\eps(y,r)|\le C.
	\)

	For the tangential coefficients, since the cutoff depends only on $r$,
	\(
	\nabla^h h_i^\eps
	=
	\chi(r/\rho_\eps)\nabla^h h_i,
	\) where \(\nabla^h\) denotes the Levi-Civita connection with respect to \(h\).
	Hence smoothness of the original metrics on the finitely many compact
	collars gives
	\(
	\sup_{\eps,i,y,r}
	|\nabla^h h_i^\eps(y,r)|_h\le C.
	\)
	The standard difference-of-Levi--Civita-connections formula gives
	\[
	|\nabla^{h_i^\eps(r)}-\nabla^h|_h
	\le
	C\,|(h_i^\eps(r))^{-1}|_h\,
	|\nabla^h h_i^\eps(r)|_h
	\le C .
	\]
	For $u\in C^2(\Gamma)$,
	\(
	\nabla^{2,h_i^\eps(r)}u
	=
	\nabla^{2,h}u-du\circ (\nabla^{h_i^\eps(r)}-\nabla^h ),
	\)
	and therefore, taking the trace with respect to $h_i^\eps(r)$ and using
	\eqref{eq:tangential-metric-uniform-equivalence},
	\(
	|\Delta_{h_i^\eps(r)}u|
	\le
	C\bigl(|\nabla^{2,h}u|_h+|du|_h\bigr).
	\)
	Consequently,
	\[
	\sup_{\eps,i,r}
	\|\Delta_{h_i^\eps(r)}\|_{C^2(\Gamma,h)\to C^0(\Gamma)}
	\le C,
	\]
	which completes \eqref{eq:fermi-uniform-coefficients}.
\end{proof}

\subsection{Smooth doubles}

For each \(i\), let \(M_i^+\) and \(M_i^-\) be two copies of \(M_i\), and define
the topological double
\[
\widetilde M_i
:=
(M_i^+\sqcup M_i^-)/{\sim},
\qquad
x^+\sim x^- \quad\text{for }x\in\Gamma .
\]
Away from the identified boundary, \(\widetilde M_i\) carries the smooth
structures of the two copies.  Across the seam, the smooth structure is
defined by the signed collar chart
\(
\widetilde\Phi_i:
\Gamma\times(-r_0,r_0)\to \widetilde M_i,
\)
where
\[
\widetilde\Phi_i(y,s)
=
\begin{cases}
	\Phi_i(y,s)\in M_i^+, & s\ge0,\\
	\Phi_i(y,-s)\in M_i^-, & s\le0.
\end{cases}
\]
With this smooth structure, \(\widetilde M_i\) is a compact smooth manifold
without boundary.

Because \(g_i^\varepsilon=dr^2+h\) for \(0\le r\le\rho_\varepsilon\),
the two copies of \(g_i^\varepsilon\) glue smoothly across the seam.  We
denote the resulting Riemannian metric on \(\widetilde M_i\) by
\(\widetilde g_i^\varepsilon\).  In the signed collar coordinates,
\[
\widetilde g_i^\varepsilon
=
ds^2+h_i^\varepsilon(y,|s|).
\]
This tensor is smooth at \(s=0\), since
\(h_i^\varepsilon(y,r)=h(y)\) for \(0\le r\le\rho_\varepsilon\).

Let
\[
\varpi_i:\widetilde M_i\longrightarrow M_i
\]
be the folding map obtained by identifying each copy with \(M_i\).  In the
signed collar,
\(
\varpi_i(y,s)=(y,|s|).
\)
\subsection{The finite-state switching mechanism}

Let $Q=(Q_{ij})_{1\le i,j\le k}$ be an irreducible Markov generator:
\[
Q_{ij}\ge0\quad(i\ne j),
\qquad
Q_{ii}=-\sum_{j\ne i}Q_{ij}.
\]
Let $\pi=(\pi_1,\dots,\pi_k)$ be its unique invariant probability vector, that is
\(
\pi^\top Q=0\) and \(
\pi_i>0.
\)

Equip $\R^k$ with
\(
\ip{u}{v}_\pi:=\sum_{i=1}^k\pi_i u_iv_i
\)
and define
\[
\mathsf H_0:=\left\{u\in\R^k:\ip{u}{\bone}_\pi=0\right\}.
\]
Write
\(
\pi_*:=\min_{1\le i\le k}\pi_i>0.
\)
Since \(k<\infty\) and every \(\pi_i>0\), the weighted norm
\(
\|u\|_\pi^2
\)
is equivalent to the usual norms on \(\R^k\). In particular,
\(
	\|u\|_\infty
	:=
	\max_{1\le i\le k}|u_i|
	\le
	\pi_*^{-1/2}\|u\|_\pi\) and \(
	\|u\|_\pi\le\|u\|_\infty.
	\)
The same comparison will be used for page-indexed vectors of
covectors or tensors, with the corresponding pointwise norm in each
component. We shall refer to this repeatedly as \emph{finite-dimensional norm equivalence}.

For $u\in\R^k$ put
\(
\cD_Q(u):=-\ip{u}{Qu}_\pi.
\)

\begin{lemma}
	\label{lem:Q-dissipation}
	For every $u\in\R^k$,
	\begin{equation}
		\cD_Q(u)
		=
		\frac12\sum_{i\ne j}\pi_iQ_{ij}(u_j-u_i)^2.
		\label{eq:DQ-edge}
	\end{equation}
	Consequently $\cD_Q(u)\ge0$, $\cD_Q(u)=0$ if and only if $u$ is constant, and there exists $\lambda_Q>0$ such that
	\begin{equation}
		\cD_Q(u)\ge\lambda_Q\norm{u}_\pi^2,
		\qquad u\in\mathsf H_0.
		\label{eq:Q-gap}
	\end{equation}
\end{lemma}

\begin{proof}
	Since the terms with \(i=j\) vanish, we may sum over all \(i,j\). Expanding,
	\begin{align*}
		\frac12\sum_{i,j}\pi_iQ_{ij}(u_j-u_i)^2
		&=
		\frac12\sum_j u_j^2\sum_i\pi_iQ_{ij}
		+
		\frac12\sum_i\pi_i u_i^2\sum_jQ_{ij}
		-\sum_{i,j}\pi_iQ_{ij}u_iu_j.
	\end{align*}
	The first term vanishes because \(\pi^\top Q=0\), while the second
	vanishes because \(Q\mathbf 1=0\). Hence
	\[
	\frac12\sum_{i,j}\pi_iQ_{ij}(u_j-u_i)^2
	=
	-\sum_i\pi_i u_i(Qu)_i
	=
	-\langle u,Qu\rangle_\pi.
	\]
	This proves \eqref{eq:DQ-edge}.
	
	All terms on the right-hand side of \eqref{eq:DQ-edge} are
	nonnegative. If \(\mathcal D_Q(u)=0\), then \(u_i=u_j\) whenever
	\(Q_{ij}>0\). Irreducibility therefore implies that \(u\) is constant.
	Conversely, constants clearly have zero energy.
	It follows that \(\mathcal D_Q\) is strictly positive on the unit
	sphere of \(\mathsf H_0\). Since this sphere is compact,
	\[
	\lambda_Q:=
	\min_{\substack{u\in\mathsf H_0\\\|u\|_\pi=1}}
	\mathcal D_Q(u)>0,
	\]
	and homogeneity gives \eqref{eq:Q-gap}.
\end{proof}


\begin{lemma}
	\label{lem:pointwise-DQ}
	There is a constant $C_Q<\infty$ such that for every $u\in\R^k$ and every state $i$,
	\[
	\sum_{j\ne i}Q_{ij}(u_j-u_i)^2
	\le C_Q\,\cD_Q(u).
	\]
\end{lemma}

\begin{proof}
	By \eqref{eq:DQ-edge},
	\[
	2\cD_Q(u)
	=
	\sum_{\ell\ne m}\pi_\ell Q_{\ell m}(u_m-u_\ell)^2
	\ge
	\pi_i\sum_{j\ne i}Q_{ij}(u_j-u_i)^2.
	\]
	Thus one may take $C_Q=2/\min_i\pi_i$.
\end{proof}

Let $q_\eps:[0,\infty)\to[0,\infty)$ be bounded and continuous for each fixed $\eps$, with
\(
\supp q_\eps\subset[0,\eps].
\)
\begin{assumption}
	\label{ass:main-scaling}
	Define
	\[
	\Theta_\eps:=\int_0^\eps q_\eps(r)\,dr,
	\qquad
	\Xi_\eps:=\int_0^\eps r^2q_\eps(r)\,dr.
	\]
	We assume
	\(
	\Theta_\eps\to\infty\) and\(
	\Xi_\eps\to0.
	\)
\end{assumption}

\subsection{Self-contained construction of the microscopic processes}

Set
\[
\widetilde E:=\bigsqcup_{i=1}^k\widetilde M_i.
\]
For a point in the doubled collar write $(i,y,s)$, with $s\in(-r_0,r_0)$.  A jump from page $i$ to page $j$ preserves $(y,s)$ and has rate
\begin{equation}
	\lambda_{ij}^\eps(y,s)
	:=
	q_\eps(|s|)Q_{ij},
	\qquad i\ne j.
	\label{eq:raw-rates}
\end{equation}
Outside $|s|\le\eps$ there are no page jumps. 

We may now define the microscopic approximating processes.
\begin{proposition}
	\label{thm:micro-existence}
	For every fixed $\eps>0$ and every starting point
	$\tilde x\in\widetilde E$, there exists a conservative càdlàg Feller, hence strong Markov, process
	$\widetilde X^\eps$ on \(\widetilde E\) with the following dynamics:
	\begin{enumerate}[label=(\roman*)]
		\item while the page label equals $i$, the continuous motion is Brownian motion on
		$(\widetilde M_i,\widetilde g_i^\eps)$;
		\item in the doubled collar, the page label jumps according to
		\eqref{eq:raw-rates}, preserving $(y,s)$.
	\end{enumerate}
	The number of page jumps on every bounded time interval is almost surely finite.  For every pagewise $C^2$ function $f=(f_i)$, the process
	\[
	f(\widetilde X_t^\eps)-f(\widetilde X_0^\eps)
	-\int_0^t\widetilde A_\eps f(\widetilde X_s^\eps)\,ds
	\]
	is a local martingale, where
	\begin{equation}
		(\widetilde A_\eps f)_i
		=
		\frac12\Delta_{\widetilde g_i^\eps}f_i
		+
		q_\eps(|s|)\sum_{j=1}^kQ_{ij}f_j,
		\label{eq:micro-generator}
	\end{equation}
	in the Fermi collar, while the sum is not present outside that collar.
\end{proposition}

\begin{proof}
	Put
	\(
	\Lambda_\eps
	:=
	\norm{q_\eps}_\infty
	\max_{1\le i\le k}(-Q_{ii})<\infty.
	\)
	Let $N^\eps$ be a rate-$\Lambda_\eps$ Poisson process, independent of all Brownian drivers.  Starting from page $i$, run Brownian motion on
	$(\widetilde M_i,\widetilde g_i^\eps)$ until the first proposal time of $N^\eps$.  If the current point is outside the switching collar, keep the same page.  If the current point is $(i,y,s)$ with $|s|\le\eps$, choose $j\ne i$ with probability
	\(\lambda_{ij}^\eps(y,s)/\Lambda_\eps
	\)
	and keep page $i$ with the remaining probability.  After an accepted page change, restart Brownian motion on the new double from the point with the same collar coordinates $(y,s)$.  Iterate this construction at the successive proposal times.
	
	Each doubled page is compact, hence its Brownian motion is conservative.  The proposal process has finitely many points on bounded intervals almost surely, so the recursive construction is defined for all times and page jumps cannot accumulate.
	
	The process is Feller.  To make this explicit, let $S_t^\eps=S_{1,t}^\eps\oplus \cdots S_{m,t}^\eps$, where \(S_{i,t}^\eps C(\widetilde M_i) \to C(\widetilde M_i)\) denotes the Feller semigroup of the Brownian motion on page \(\widetilde{M}_i\).  Define the jump operator
	\(
	(B_\eps f)(i,y,s)
	:=
	\sum_{j\ne i}\lambda_{ij}^\eps(y,s)
	\bigl(f(j,\phi_{ij}(y,s))-f(i,y,s)\bigr),
	\)
	in the switching collar, where \(\phi_{ij}:(y,s)\mapsto (y,s)\); outside that collar the summand is defined to be zero.  Because $q_\eps$ is continuous, vanishes at the edge of its support, \(f\) is continuous and the maps \(\phi_{ij}\) are smooth, $B_\eps$ maps $C(\widetilde E)$ to itself and
	\(
	\norm{B_\eps}\le 2\Lambda_\eps.
	\)
	The transition semigroup of the interlaced process is therefore given by the uniformly convergent Dyson--Phillips expansion
	\[
	T_t^\eps
	=
	S_t^\eps
	+
	\sum_{m\ge1}
	\int_{0<s_1<\cdots<s_m<t}
	S_{s_1}^\eps B_\eps
	S_{s_2-s_1}^\eps B_\eps\cdots
	B_\eps S_{t-s_m}^\eps
	\,ds_1\cdots ds_m .
	\]
	Indeed, the $m$th term has operator norm at most
	$\norm{B_\eps}^m t^m/m!$.  Hence $T_t^\eps$ preserves $C(\widetilde E)$ and is strongly continuous at $t=0$.  It is the semigroup of the Poisson-thinning construction above, so $\widetilde X^\eps$ is Feller and therefore strong Markov.
	
	Conditional on the continuous path, thinning the proposal process gives compensator
	$\lambda_{ij}^\eps\,dt$ for $i\to j$ jumps.  It\^o's formula between proposal times and the compensated-jump formula at accepted transitions therefore give \eqref{eq:micro-generator} and the stated martingale relation.
\end{proof}


\subsection{Folding and the observable process}

Let
\[
E^\sharp=\bigsqcup_{i=1}^kM_i
\]
be the separated page space.  Let
\(
\varpi:\widetilde E\longrightarrow E^\sharp
\)
denote the pagewise folding map, whose restriction to $\widetilde M_i$ is $\varpi_i$.
Fold each double by
$\varpi_i$ and denote the resulting process by
\(
Y^\eps_t=(I^\eps_t,X^\eps_t)\in E^\sharp.
\)
Its quotient observation on the glued open book is
\(
Z_t^\eps:=\mathfrak q(Y_t^\eps)\in\cM.
\)

\begin{proposition}
	\label{prop:folded-generator}
	Let $(T_n^\eps)_{n\ge0}$ be the successive page-switching times.
	On the event
	\(
	\{T_n^\eps<\infty,\ I_{T_n^\eps}^\eps=i\},
	\)
	the restriction of $Z^\eps$ to
	$[T_n^\eps,T_{n+1}^\eps)$ is normally reflected Brownian motion on
	$(M_i,g_i^\eps)$, stopped at $T_{n+1}^\eps$.
	In the Fermi collar, a switch from page $i$ to page $j$ preserves
	$(y,r)$ and occurs at rate
	\(
	q_\eps(r)Q_{ij}.
	\)
	
	For every pagewise $C^2$ function
	$f=(f_1,\ldots,f_k)$ satisfying
	\(
	\partial_{\nu_i}f_i|_\Gamma=0\) for every \( i=1,\ldots,k,
	\)
	the process
	\[
	f(Y_t^\eps)-f(Y_0^\eps)
	-
	\int_0^t A_\eps f(Y_s^\eps)\,ds
	\]
	is a local martingale, where, in the Fermi collar,
	\[
	(A_\eps f)_i(y,r)
	=
	\frac12\Delta_{g_i^\eps}f_i(y,r)
	+
	q_\eps(r)\sum_{j=1}^kQ_{ij}f_j(y,r),
	\]
	while outside the switching collar
	\[
	(A_\eps f)_i=\frac12\Delta_{g_i^\eps}f_i.
	\]
\end{proposition}

%

\begin{proof}
	For each $i$, let
	\(
	R_i:\widetilde M_i\to\widetilde M_i
	\)
	be the involution exchanging the two copies of $M_i$.  In the doubled
	Fermi collar it is given by
	\(
	R_i(y,s)=(y,-s).
	\)
	Since
	\(
	\widetilde g_i^\varepsilon
	=
	ds^2+h_i^\varepsilon(y,|s|),
	\)
	the map $R_i$ is an isometry.  Hence Brownian motion on
	$(\widetilde M_i,\widetilde g_i^\varepsilon)$ is invariant under $R_i$.
	
	The switching mechanism has the same symmetry.  Indeed, its rates depend
	on the signed normal coordinate only through $|s|$, and a switch
	$i\to j$ sends $(i,y,s)$ to $(j,y,s)$.  Thus reflecting before or after a
	switch gives the same result, and the law of the folded process starting
	from $\widetilde x\in\widetilde E$ depends only on
	$\varpi(\widetilde x)$.  Consequently the transition kernel of
	$\widetilde X^\varepsilon$ descends through $\varpi$ to a transition
	kernel on $E^\sharp$.
	
	If $\tau$ is a stopping time
	for the filtration of $Y^\varepsilon$, then it is also a stopping time
	for that of $\widetilde X^\varepsilon$.  Then, if \(\widetilde{P}_t^\eps\) and \(P_t^\eps\) denote the transition semigroups of \(\widetilde X^\eps\) and \(Y^\eps\) respectively, for every bounded measurable
	$f$ on $E^\sharp$, we have
	\(
	\mathbb E_{\widetilde x}[f(Y^\varepsilon_{\tau+t})\mid\mathcal F^{Y^\varepsilon}_\tau]
	=
	\mathbb E_{\widetilde x}
	[
	\widetilde P_t^\varepsilon(f\circ\varpi)
	(\widetilde X^\varepsilon_\tau)
	\,\Bigm|\,
	\mathcal F^{Y^\varepsilon}_\tau]  =
	P_t^\varepsilon f(Y^\varepsilon_\tau),
	\)
	where the last equality uses the fact that
	$\widetilde P_t^\varepsilon(f\circ\varpi)$ is constant on the fibres of
	$\varpi$. Thus \(Y^\eps\) is strong Markov as well.
	
	It remains to identify the generator.  Let
	$f=(f_1,\ldots,f_k)$ be pagewise $C^2$ and satisfy
	\(
	\partial_{\nu_i}f_i|_\Gamma=0 .
	\)
	Lift $f_i$ to the double by
	\(
	\widetilde f_i:=f_i\circ\varpi_i .
	\)
	In doubled Fermi coordinates,
	\(
	\widetilde f_i(y,s)=f_i(y,|s|).
	\)
	The Neumann condition implies that the one-sided first derivatives of \(\widetilde f\) are both null at $s=0$; since $f_i$ is $C^2$ up to the boundary and 
	$\widetilde f_i$ is even, its second one-sided derivatives agree as well and it is $C^2$ across the seam.  Moreover, on either side of
	the seam the folding map is a local isometry, and therefore
	\(
	\Delta_{\widetilde g_i^\varepsilon}\widetilde f_i(y,s)
	=
	\bigl(\Delta_{g_i^\varepsilon}f_i\bigr)(y,|s|).
	\)
	By continuity, the same identity holds at $s=0$.
	
	Applying Proposition~\ref{thm:micro-existence} to $\widetilde f=f\circ\varpi$
	therefore gives, in the doubled collar,
	\begin{equation}
		\label{eq:folded-generator}
		\begin{aligned}
			(\widetilde A_\varepsilon\widetilde f)_i(y,s)
			&=
			\frac12
			\Delta_{\widetilde g_i^\varepsilon}\widetilde f_i(y,s)
			+
			q_\varepsilon(r)
			\sum_{j=1}^k Q_{ij}\widetilde f_j(y,s)=
			\frac12\Delta_{g_i^\varepsilon}f_i(y,|s|)
			+
			q_\varepsilon(|s|)\sum_{j=1}^kQ_{ij}f_j
			(y,|s|).
		\end{aligned}
	\end{equation}
	Equivalently,
	\(
	\widetilde A_\varepsilon(f\circ\varpi)
	=
	(A_\varepsilon f)\circ\varpi,
	\)
	with
	\(
	(A_\varepsilon f)_i
	=
	\frac12\Delta_{g_i^\varepsilon}f_i
	+
	q_\varepsilon(r)(Qf)_i
	\)
	in the switching collar.  Outside that collar the switching rates vanish,
	so only the Laplace--Beltrami term remains.
	
	On the event
	\(
	\{T_n^\eps<\infty,\ I_{T_n^\eps}^\eps=i\},
	\)
	no page switch occurs on $[T_n^\eps,T_{n+1}^\eps)$.  Hence, for every
	$f_i\in C^2(M_i)$ satisfying
	$\partial_{\nu_i}f_i|_\Gamma=0$, \eqref{eq:folded-generator} gives
	\[
	f_i\!\left(X_{t\wedge T_{n+1}^\eps}^\eps\right)
	-f_i\!\left(X_{T_n^\eps}^\eps\right)
	-\frac12
	\int_{T_n^\eps}^{t\wedge T_{n+1}^\eps}
	\Delta_{g_i^\eps}f_i(X_s^\eps)\,ds
	\]
	as a local martingale, for $t\ge T_n^\eps$.
	Thus, up to the stopping time $T_{n+1}^\eps$, the folded spatial motion
	solves the Neumann martingale problem for
	$\frac12\Delta_{g_i^\eps}$ on $M_i$, and hence is normally reflected
	Brownian motion on $(M_i,g_i^\eps)$ stopped at $T_{n+1}^\eps$.
	The same statement holds for its quotient observation $Z^\eps$ while
	the active page is $i$.
\end{proof}

\section{Uniform tightness from deterministic starting points}
\label{sec:tightness}

We next prove \(C\)-tightness uniformly over all microscopic starting points.

\begin{lemma}
	\label{lem:localtime-estimate}
	For every $T<\infty$ there exist $\eps_0>0$ such that,
	for every $0<\eps\le\eps_0$, there is a continuous
	$(\mathcal F_t^{\widetilde X^\eps})$-adapted real semimartingale $R^\eps$ satisfying
	$R_t^\eps=r$ whenever $Y_t^\eps=(i,y,r)$ and $r<3r_0/4$, and admitting
	the decomposition
	\[
	R_t^\eps=R_0^\eps+M_t^\eps+A_t^\eps+K_t^\eps .
	\]
	Here $M^\eps$ is a continuous local martingale, $A^\eps$ has absolutely
	continuous paths, and $dK^\eps$ is carried by $\{R^\eps=0\}$. Moreover,
	uniformly in $0<\eps\le\eps_0$, the page label and the initial state,
	\begin{equation}
		\frac{d}{dt}\langle M^\eps\rangle_t\le C,
		\qquad
		\left|\frac{d}{dt}A_t^\eps\right|\le C,
		\label{eq:global-radial-bounds}
	\end{equation}
	while
	\begin{equation}
		\frac{d}{dt}\langle M^\eps\rangle_t=1
		\qquad\text{whenever }0\le R_t^\eps \le \frac12 r_0.
		\label{eq:unit-radial-qv}
	\end{equation}
	For every
	$(\mathcal F_t^\eps)$-stopping time $\tau\le T$, every
	$0<\delta\le1$, and every $0\le a<b\le r_0/2$, there exists $C_T<\infty$ such that
	\begin{equation}
		\label{eq:occupation-interval}
		\sup_{\tilde x\in\widetilde E}
		\E_{\tilde x}\int_\tau^{(\tau+\delta)\wedge T}
		\mathbf1_{\{a<R_s^\eps<b\}}\,ds
		\le C_T(b-a)\sqrt\delta.
	\end{equation}
	Moreover, for every nonnegative Borel function $\varphi$
	supported in $[0,r_0/2]$,
	\begin{equation}
		\label{eq:occupation-L1}
		\sup_{\tilde x\in\widetilde E}
		\E_{\tilde x}\int_0^T\varphi(R_s^\eps)\,ds
		\le
		C_T\int_0^{r_0/2}\varphi(r)\,dr .
	\end{equation}
\end{lemma}

\begin{proof}
	Choose a smooth function
	\(
	\rho:M_i\to[0,\infty)
	\)
	on every page, using the same construction for all $i$, such that
	$\rho(y,r)=r$ for $0\le r\le3r_0/4$, $\rho$ is constant outside a slightly
	larger collar, and $\rho^{-1}(\{0\})=\Gamma$.  Because the page-switching jumps
	preserve $(y,r)$ and occur inside $r\le\eps<3r_0/4$, the scalar process
	\(
	R_t^\eps:=\rho(X_t^\eps)
	\)
	is continuous across every page switch.
	
	On an interval on which the page label is $i$, reflected It\^o's
	formula gives
	\[
	dR_t^\eps
	=
	dM_t^\eps
	+
	\frac12\Delta_{g_i^\eps}\rho_i(X_t^\eps)\,dt
	+
	\partial_{\nu_i}\rho_i(X_t^\eps)\,dK_t^\eps,
	\]
	where $M^{\eps,n}$ is a continuous local martingale satisfying
	\(
	d\langle M^{\eps,n}\rangle_t
	=
	|\nabla\rho|_{g_i^\eps}^2(X_t^\eps)\,dt,
	\)
	and $K^{\eps,n}$ is the reflection regulator.
	Since $dK^\eps$ is supported on $\Gamma$ and \(\rho=r\) near \(\Gamma\),
	$\partial_{\nu_i}\rho_i=1$ there. Therefore, the last term is $dK_t^\eps$.
	At every page switch the collar coordinates $(y,r)$ are preserved and
	$\rho_i=\rho_j=r$ on the switching region, so $R^\eps$ has no jump.
	Summing the identities over the finitely many inter-switch intervals
	therefore yields
	\[
	R_t^\eps
	=
	R_0^\eps
	+
	M_t^\eps
	+
	A_t^\eps
	+
	K_t^\eps,
	\]
	where 
	$A_t^\eps
	:=
	\/2\int_0^t
	\Delta_{g_{I_s^\eps}^\eps}
	\rho_{I_s^\eps}(X_s^\eps)\,ds$.    Moreover,
	$dK^\eps$ is increases only on $\{R^\eps=0\}$.
	
	The functions $\rho_i$ have uniformly bounded first derivatives, and \eqref{eq:metric-uniform-equivalence},
	gives
	\(
	\frac{d}{dt}\langle M^\eps\rangle_t=|\nabla\rho_{I_t^\eps}|_{g_{I_t^\eps}^\eps}^2(X_t^\eps)\le C,
	\) uniformly in \(i\), \(\eps\) and \(\tilde x.\)
	
	For the drift we distinguish two regions.  On the collar on which
	$\rho_i=r$, the Fermi representation
	$g_i^\eps=dr^2+h_i^\eps(y,r)$ gives
	\(
	\Delta_{g_i^\eps}\rho_i
	=
	\Delta_{g_i^\eps}r
	=
	\partial_r\log J_i^\eps.
	\)
	By \eqref{eq:fermi-uniform-coefficients},
	$\partial_r\log J_i^\eps$ is uniformly bounded.  Away from this collar, $g_i^\eps=g_i$, while $\rho_i$ and
	$g_i$ are smooth on a compact set.  Hence
	$\Delta_{g_i^\eps}\rho_i$ is uniformly bounded there as well.
	Consequently,
	\(|\frac{d}{dt}A_t^\eps|\le C,\)
	with $C$ independent of $0<\eps\le\eps_0$, of the page label and of the
	initial state, giving \eqref{eq:global-radial-bounds}.
	On the smaller collar
	$R^\eps<r_0/2$ we have $\rho=r$, which gives \eqref{eq:unit-radial-qv}

	Since $R^\eps$ is the one-dimensional
	Skorokhod reflection at the origin of the continuous driving path
	$R_0^\eps+M^\eps+A^\eps$, applying the Lipschitz estimate for the
	Skorokhod map to the increments after $\tau$ together with \eqref{eq:global-radial-bounds} and Burkolder-Davis-Gundy (BDG) inequality give, with
	$\theta_\delta=(\tau+\delta)\wedge T$ and similarly for \(\theta_u\),
	\begin{gather}
		\sup_{0\le u\le\delta}
		\left|
		R^\eps_{\theta_u}-R^\eps_\tau
		\right|
		\le
		C
		\sup_{0\le u\le\delta}
		\left|
		M^\eps_{\theta_u}-M^\eps_\tau
		\right|
		\le C_T\sqrt \delta, \label{eq:radial-modulus}\\
		K^\eps_{\theta_\delta}-K^\eps_\tau \le \sup_{0\le u\le\delta}
		\left|
		A^\eps_{\theta_u}-A^\eps_\tau
		\right| \le C \delta. \label{eq:K-estimate}
	\end{gather}
	
	%

	Fix $a\in(0,r_0/2]$.  Tanaka's formula for $|R^\eps-a|$ between
	$\tau$ and $\theta_\delta$ gives
	\begin{align*}
		L_{\theta_\delta}^a(R^\eps)-L_\tau^a(R^\eps)
		&=
		|R_{\theta_\delta}^\eps-a|-|R_\tau^\eps-a|
		-
		\int_\tau^{\theta_\delta}
		\sgn(R_s^\eps-a)\,dM_s^\eps
		\\
		&\quad
		-
		\int_\tau^{\theta_\delta}
		\sgn(R_s^\eps-a)\,dA_s^\eps
		-
		\int_\tau^{\theta_\delta}
		\sgn(R_s^\eps-a)\,dK_s^\eps .
	\end{align*}
	Since $a>0$ and $dK^\eps$ is carried by $\{R^\eps=0\}$,
	$\sgn(R_s^\eps-a)=-1$ on the support of $dK^\eps$, so
	\(
	-
	\int_\tau^{\theta_\delta}
	\sgn(R_s^\eps-a)\,dK_s^\eps
	=
	K_{\theta_\delta}^\eps-K_\tau^\eps.
	\)
	Moreover,
	\(
	\left|
	|R_{\theta_\delta}^\eps-a|-|R_\tau^\eps-a|
	\right|
	\le
	|R_{\theta_\delta}^\eps-R_\tau^\eps|.
	\)
	Using
	\eqref{eq:global-radial-bounds}, \eqref{eq:radial-modulus}, \eqref{eq:K-estimate} and BDG for the martingale integral,
	we obtain
	\begin{equation}
		\\
		\E_{\tilde x}\!\left[
		L^a_{(\tau+\delta)\wedge T}(R^\eps)-L^a_\tau(R^\eps)
		\right]
		\le C_T\sqrt\delta. \label{eq:local-time-bound}
	\end{equation}
	
	For $0<a<b<r_0/2$, \eqref{eq:unit-radial-qv} and the occupation-density formula therefore yield
	\[
	\int_\tau^\theta
	\mathbf1_{\{a<R_s^\eps<b\}}\,ds
	=
	\int_\tau^\theta
	\mathbf1_{\{a<R_s^\eps<b\}}\,
	d\langle M^\eps\rangle_s=
	\int_a^b
	\left(
	L_\theta^r(R^\eps)-L_\tau^r(R^\eps)
	\right)\,dr.
	\]
	Taking expectations and applying
	\eqref{eq:local-time-bound} proves
	\eqref{eq:occupation-interval} for $0<a<b<r_0/2$.
	The endpoint cases follow by monotone convergence.  In particular,
	there is no Lebesgue-time atom at the boundary.  Indeed,
	\eqref{eq:unit-radial-qv} also holds at $R^\eps=0$ and \(R^\eps=r_0/2\), so the
	occupation-density formula gives
	\(
	\int_0^T
	\mathbf1_{\{R_s^\eps=0\}}\,ds
	=
	\int_0^T
	\mathbf1_{\{R_s^\eps=r_0/2\}}\,ds
	=
	0 
	\) almost surely. Finally, cover $[0,T]$ by finitely many deterministic intervals of
	length at most one and apply \eqref{eq:occupation-interval} on each
	one.  Absorbing the number of intervals into $C_T$, we obtain
	\[
	\sup_{\tilde x}
	\E_{\tilde x}
	\int_0^T
	\mathbf1_{\{a<R_s^\eps<b\}}\,ds
	\le
	C_T(b-a),
	\qquad
	0\le a<b\le r_0/2,
	\]
	uniformly for $0<\eps\le\eps_0$.
	The preceding interval estimate, together with
	$\mu_{\eps,\tilde x}^T(\{0\})=0$, implies that the expected occupation measure satisfies
	\[
	\mu_{\eps,\tilde x}^T(B)
	:=
	\E_{\tilde x}
	\int_0^T
	\mathbf1_{\{R_s^\eps\in B\}}\,ds.
	\le
	C_T\,\operatorname{Leb}(B).
	\]
	Hence, for every
	nonnegative Borel function $\varphi$ supported in $[0,r_0/2]$,
	\[
	\sup_{\tilde x}
	\E_{\tilde x}
	\int_0^T
	\varphi(R_s^\eps)\,ds
	=
	\sup_{\tilde x}
	\int_{[0,r_0/2]}
	\varphi(r)\,\mu_{\eps,\tilde x}^T(dr)
	\le
	C_T\int_0^{r_0/2}\varphi(r)\,dr,
	\]
	uniformly for $0<\eps\le\eps_0$.
\end{proof}

Let \(\mathcal{L}_x(X)\) denote the law of the process \(X\) started at the point \(x\).

\begin{lemma}
	\label{lem:C-tightness}
	Fix $T<\infty$.  The family
	\[
	\left\{\mathcal L_{\tilde x_\eps}(Z^\eps):
	\eps>0,\ \tilde x_\eps\in\widetilde E\right\}
	\]
	is tight in $D([0,T],\cM)$, and every subsequential limit is supported on
	$C([0,T],\cM)$.
\end{lemma}

\begin{proof}
	We prove tightness using Aldous' criterion, but, before applying it, we remove the high jump frequency completely.  Choose a nondecreasing $\eta\in C^\infty([0,\infty);[0,1])$ such that
	$\eta(s)=0$ for $s\le1$ and $\eta(s)=1$ for $s\ge2$, and define
	$\rho_\eps(r):=r\,\eta(r/\eps)$. Then
	$\rho_\eps(r)=0$ for $r\le\eps$ and $\rho_\eps(r)=r$ for
	$r\ge2\eps$. Moreover, for a constant $C$ independent of $\eps$,
	$|\rho_\eps(r)-r|\le C\eps$, $0\le\rho_\eps'(r)\le C$, and
	$|\rho_\eps''(r)|\le C/\eps$, with
	$\operatorname{supp}\rho_\eps''\subset(\eps,2\eps)$.
	
	Let $F_\eps:\cM\to \cM$ move points in the collar along the normal geodesics
	according to $r\mapsto\rho_\eps(r)$ and be the identity outside the
	collar. Set
	$\bar Z^\eps:=F_\eps(Z^\eps)$. 
	
	We first check that the rapid label process produces no hidden term in
	$\bar Z^\eps$.  In an interface chart use the embedding
	\(
	(i,y,r)\longmapsto(y,re_i)\in\R^{n-1+k},
	\) where the \(e_i\)'s are the standard basis vectors of \(\R^k\),
	and write the embedded collapsed normal component as
	\(
	V_t^\eps:=\rho_\eps(R_t^\eps)e_{I_t^\eps}.
	\)
	At every switching time $t$, $R_t^\eps\le\eps$, hence
	$\rho_\eps(R_t^\eps)=0$.  The product formula for a continuous semimartingale
	and the pure-jump process $e_{I^\eps}$ therefore has
	\begin{equation}
		\rho_\eps(R_{t-}^\eps)\Delta e_{I_t^\eps}=0
		\label{eq:V-jump}
	\end{equation}
	at every label jump.  Thus $V^\eps$ is continuous and, between label jumps, the one-dimensional
	It\^o's formula gives
	\begin{equation}
		\begin{array}{r c c c c c c}
			dV_t^\eps
			&=&
			e_{I_t^\eps}\rho_\eps'(R_t^\eps)\,dM_t^\eps
			&+&
			e_{I_t^\eps}\rho_\eps'(R_t^\eps)\,dA_t^\eps
			&+&
			e_{I_t^\eps}\frac12\rho_\eps''(R_t^\eps)\,
			d\langle M^\eps\rangle_t
			\\[1ex]
			&=:&
			dN_t^\eps
			&+&
			dB_t^\eps
			&+&
			dD_t^\eps .
		\end{array}
		\label{eq:collapsed-radial-vector}
	\end{equation}
	There is no \(dK^\eps_t\)-term because
	$\rho_\eps'(0)=0$, and there is no switching compensator because of \eqref{eq:V-jump}.
	
	For every \((\mathcal F_t^{\widetilde X^\eps})\)-stopping time \(\tau\), writing \(\theta_\delta=(\tau+\delta)\wedge T\), equation \eqref{eq:global-radial-bounds} and \(\rho_\eps'\le C\) give, \[\frac{d}{dt}\langle N^\eps\rangle_t\le C\qquad|B_{\theta_\delta}^\eps-B_\tau^\eps|\le C\delta.\]
	For the \(D_t^\eps\) term, equations \eqref{eq:unit-radial-qv} and \eqref{eq:occupation-interval}, $|\rho_\eps''(r)|\le C/\eps$ and
	$\operatorname{supp}\rho_\eps''\subset(\eps,2\eps)$ give
	\begin{align}
		&\E\int_\tau^{(\tau+\delta)\wedge T}
		|\rho_\eps''(R_s^\eps)|\,d\langle M^\eps\rangle_s
		\le
		\frac{C}{\eps}
		\E\int_\tau^{(\tau+\delta)\wedge T}
		\mathbf1_{\{\eps<R_s^\eps<2\eps\}}\,ds
		\le C_T\sqrt\delta.
		\label{eq:collapse-singular-drift}
	\end{align}
	Applying BDG to \(N_t^\eps\), we get
	\(\E_{\tilde x}[|V_{\theta_\delta}^\eps-V_\tau^\delta|]\le C_T(\sqrt \delta +\delta)\)
	which in turn yields Aldous's condition: \[\Pp_{\tilde x}(|V_{\theta_\delta}^\eps-V_\tau^\delta|>\gamma)\to 0,\quad\text{for }\gamma>0\].
	
	We next control the tangential coordinates.  Fix  relatively compact
	coordinate nested charts $U_0\Subset U_1\Subset\Gamma$, with coordinates
	$y=(y^1,\ldots,y^{n-1})$ defined on \(U_1\), and choose $0<r_1<r_0$.  The corresponding
	interface neighborhoods in the glued space are
	\[
	\mathcal U_i
	:=
	\mathfrak q\left(
	\bigsqcup_{i=1}^k
	\Phi_i(U_i\times[0,r_1))\right
	), \qquad i=1,2.
	\]
	For $a=1,\ldots,n-1$, the common tangential coordinate defines a
	continuous function $\mathsf y^a:\mathcal U_1\to\R$ by
	\(
	\mathsf y^a\bigl(\mathfrak q(\Phi_i(y,r))\bigr):=y^a .
	\)
	This is well defined at the binding because the boundary copies of
	$\Gamma$ are identified through the same coordinate $y$.  Let
	\(
	\sigma^\eps:=\inf\{t\ge0:Z_t^\eps\notin\mathcal U_0\}
	\)
	and set
	\(
	Z_t^{a,\eps}:=\mathsf y^a(Z_t^\eps),\) for \( t<\sigma^\eps.
	\)
	Since the collapse map $F_\eps$ changes only the normal
	coordinate,
	\(
	\mathsf y^a(\bar Z_t^\eps)=\mathsf y^a(Z_t^\eps),
	\)
	so it suffices to estimate $Z^{a,\eps}$.
	
	By Proposition~\ref{prop:folded-generator}, between successive
	page-switching times, on page $i$,
	$Z^\eps$ is normally reflected Brownian motion on $(M_i,g_i^\eps)$.
	Applying It\^o's formula to the tangential coordinate $y^a$ and
	concatenating across the switching times gives
	\begin{equation}
		Z_{t\wedge\sigma^\eps}^{a,\eps}
		=
		Z_0^{a,\eps}
		+
		N_{t\wedge\sigma^\eps}^{a,\eps}
		+
		\frac12
		\int_0^{t\wedge\sigma^\eps}
		\Delta_{g_{I_s^\eps}^\eps}y^a(Z_s^\eps)\,ds ,
		\label{eq:tangential-decomposition}
	\end{equation}
	where $N^{a,\eps}$ is a continuous local martingale.  There are no
	additional terms in \eqref{eq:tangential-decomposition}: a page switch
	preserves $(y,r)$ and hence produces no jump in $Z^{a,\eps}$, while
	normal reflection produces no regulator term because
	$\partial_r y^a=0$.
	
	By \eqref{eq:tangential-metric-uniform-equivalence},
	\[
	\frac{d}{dt}\langle N^{a,\eps}\rangle_t
	=
	(h_{I_t^\eps}^\eps)^{-1}(dy^a,dy^a)
	\le
	C h^{-1}(dy^a,dy^a)
	\le C
	\]
	before $\sigma^\eps$.
	
	Since $y^a$ is independent of $r$,
	\(
	\Delta_{g_i^\eps}y^a
	=
	\Delta_{h_i^\eps(r)}y^a .
	\)
	After extending \(y^a\) smoothly to \(\Gamma\) with a cutoff \(\chi \in C_c^\infty(U_1), \ \chi \equiv 1\) on \(U_0\), by \eqref{eq:fermi-uniform-coefficients}, \(\Delta_{h_i^\eps(r)}y^a \le C_a\) on \(\mathcal U_0\).
	
	Now, similar estimates to the normal direction give Aldous' condition for $\mathsf{y}_a(\bar Z^\eps)$ on the interface chart.
	
	On an interior chart whose closure is disjoint from $\Gamma$, for all
	sufficiently small $\eps$ one has $g_i^\eps=g_i$ and no switching
	occurs, so the same stopped diffusion estimate gives the same
	condition directly.
	
	Now a localization agument lets us pass to the global Aldous' condition
	\[
	\lim_{\delta\downarrow0}
	\sup_{0<\varepsilon\le\varepsilon_0}
	\sup_{\tilde x\in\widetilde E}
	\sup_{\tau\le T}
	\mathbb P_{\tilde x}
	\left(
	d_{\mathcal{M}}\!\left(
	\bar Z^\varepsilon_{(\tau+\delta)\wedge T},
	\bar Z^\varepsilon_\tau
	\right)>\gamma
	\right)
	=0, \quad \gamma>0.
	\]

	Since \(\mathcal{M}\) is compact, the compact-containment condition is automatic.
	Aldous' criterion therefore yields tightness of
	\(\{\bar Z^\varepsilon\}\) in \(D([0,T],\cM)\), uniformly over the
	microscopic starting point.  Moreover, every \(\bar Z^\varepsilon\)
	has continuous paths: the continuous diffusion and reflection parts
	remain continuous, while every page switch occurs with
	\(R^\varepsilon_t\le\varepsilon\), where
	\(\rho_\varepsilon(R^\varepsilon_t)=0\), and hence produces no jump
	after applying \(F_\varepsilon\).  Since \(C([0,T],\cM)\) is closed in
	the Skorokhod \(J_1\) topology, every weak subsequential limit of
	\(\bar Z^\varepsilon\) is supported on \(C([0,T],\cM)\).  Thus
	\(\bar Z^\varepsilon\) is \(C\)-tight.
	
	Finally, by the definition of \(F_\varepsilon\),
	\(
	\sup_{x\in \cM} d_\cM(F_\varepsilon x,x)\le C\varepsilon,
	\)
	and therefore, pathwise,
	\(
	\sup_{t\le T}
	d_\cM(Z^\varepsilon_t,\bar Z^\varepsilon_t)
	\le C\varepsilon.
	\)
	In particular,
	\[
	d_{J_1}(Z^\varepsilon,\bar Z^\varepsilon)
	\le
	\sup_{t\le T}
	d_\cM(Z^\varepsilon_t,\bar Z^\varepsilon_t)
	\to0 .
	\]
	Hence \(Z^\varepsilon\) and \(\bar Z^\varepsilon\) have the same
	subsequential weak limits.  Since the latter family is \(C\)-tight,
	the family \(Z^\varepsilon\) is \(C\)-tight as well.
\end{proof}

\section{The non-symmetric boundary-layer corrector}
\label{sec:corrector}

The key step is a corrector that removes the reflected-boundary local-time term and simultaneously cancels the first-order jump drift.  This is where irreducibility, rather than reversibility, enters.

On the Hilbert space
\(
V_\eps:=H^1((0,2\eps);\mathsf H_0),
\)
define the bilinear form
\[
B_\eps(v,w)
:=
\frac12\int_0^{2\eps}\ip{v'}{w'}_\pi\,dr
-\int_0^\eps q_\eps(r)\ip{Qv}{w}_\pi\,dr.
\]
Let
\(
\delta_\eps:=2\eps+\Theta_\eps^{-1}.
\label{eq:deltaeps}
\)
Extend $q_\eps$ by zero on $(\eps,2\eps]$.

\begin{lemma}
	\label{lem:corrector}
	Let $a\in\mathsf H_0$.  There is a unique
	\(
	v^\eps_a\in H^1((0,2\eps);\mathsf H_0)
	\)
	satisfying, weakly,
	\begin{equation}
		\frac12(v^\eps_a)''(r)
		+q_\eps(r)Qv^\eps_a(r)
		=
		-r q_\eps(r)Qa,
		\qquad 0<r<2\eps,
		\label{eq:corrector-ode}
	\end{equation}
	with Neumann data \((v^\eps_a)'(0)=-a,\) and \((v^\eps_a)'(2\eps)=0.\)
	Moreover there is a constant $C$, depending only on $Q$ and $\pi$, such that
	\begin{gather}
		B_\eps(v,w)\le
		C\norm{a}_\pi^2\bigl(\delta_\eps+\Xi_\eps\bigr),
		\label{eq:corrector-energy}\\
		\sup_{0\le r\le2\eps}\norm{v^\eps_a(r)}_\pi
		\le
		C\norm{a}_\pi
		\sqrt{\delta_\eps}\,
		\bigl(\sqrt{\delta_\eps}+\sqrt{\Xi_\eps}\bigr),
		\label{eq:corrector-sup}\\
		\int_0^{2\eps}\norm{(v^\eps_a)'(r)}_\pi\,dr
		\le
		C\norm{a}_\pi\sqrt{\eps}\,
		\bigl(\sqrt{\delta_\eps}+\sqrt{\Xi_\eps}\bigr).
		\label{eq:corrector-L1der}
	\end{gather}
	In particular, under Assumption \ref{ass:main-scaling}, all three quantities on the right of
	\eqref{eq:corrector-energy}--\eqref{eq:corrector-L1der} tend to zero.
\end{lemma}

\begin{proof}
	For fixed $\eps$, this is coercive on $V_\eps$ with respect to the usual $H^1$ norm.  To show this, write
	\begin{equation}
		\label{eq:averaging}
		\norm{v(0)}_\pi
		\le \norm{v(r)}_\pi+\sqrt{2\eps}\norm{v'}_{L^2(0,2\eps;\pi)}.
	\end{equation}
	Taking the square and averaging against the probability measure $q_\eps(r)\,dr/\Theta_\eps$ gives
	\begin{equation}
		\label{eq:v0-estimate}
		\norm{v(0)}_\pi^2\le 2\Theta_\eps^{-1}\int_0^\eps q_\eps(r)\|v(r)\|_\pi^2\, dr +4\eps\int_0^{2\eps}\|v'(r)\|_\pi^2\,dr
		\le (2\lambda_Q^{-1}\Theta_\eps^{-1}+8\eps) B_\eps(v,v),
	\end{equation}
	where in the second inequality we have used \eqref{eq:Q-gap}.
	Then
	\[
	\int_0^{2\eps}\norm{v(r)}_\pi^2\,dr
	\le
	4\eps\norm{v(0)}_\pi^2
	+8\eps^2\int_0^{2\eps}\norm{v'(r)}_\pi^2\,dr
	\le C_\eps B_\eps(v,v).
	\]
	Thus $B_\eps(v,v)\ge C_\eps^{-1}\norm{v}_{H^1}^2$. Because \(H_0\) is finite-dimensional, \(Q\) is  bounded linear operator. Since $q_\eps$ is bounded, the bilinear form is then bounded.  Hence Lax--Milgram applies to the weak formulation of \eqref{eq:corrector-ode}, that is
	\begin{equation}
		B_\eps(v,w)
		=
		\frac12\ip{a}{w(0)}_\pi
		+
		\int_0^\eps r q_\eps(r)\ip{Qa}{w(r)}_\pi\,dr.
		\label{eq:corrector-weak}
	\end{equation}
	Because \(r q_\eps(r)\) is bounded for fixed \(\eps\), denoting \(\ell_\eps(w)\) the left hand side in \eqref{eq:corrector-weak}, we have
	\begin{equation}
		\label{eq:lhs-estimate}
		\ell_\eps(w) \le \|a\|_\pi \|w(0)\|_\pi+\|Qa\|_\pi\int_0^\eps rq_\eps(r)\|w(r)\|_\pi \,dr \le C_\eps  \|w\|_{H^1}
	\end{equation}
	Both \(q_\eps Q v\) and \(r q_\eps Q a\) are in \(L^2\). Hence \(v'' \in L^2\) and $v \in H^2$.  Taking tast functions compactly supported in \([0,2\eps]\), integration by parts identifies \eqref{eq:corrector-ode} and
	the Neumann data.
	
	We now derive estimates uniform in $\eps$. 
	
	Taking $w=v$ in \eqref{eq:corrector-weak}, \eqref{eq:lhs-estimate} gives
	\[
	B_\eps(v,v) \le \|a\|_\pi\|v(0)\|_\pi+ C\|a\|_\pi\int_0^\eps rq_\eps(r)\|w(r)\|_\pi \, dr.
	\]
	The bound \eqref{eq:v0-estimate} gives \(\|v(0)\|_\pi\le C \sqrt{\delta_\eps B_\eps(v,v)}\). 
	For the second term, write \(rq_\eps \|v\|_\pi =(r \sqrt{q_\eps})(\sqrt{q_\eps}\|v\|_\pi)\). Cauchy-Schwartz then gives 
	\[\int_0^\eps rq_\eps(r)\|w(r)\|_\pi \, dr \le \left(\int_0^\eps r^2q_\eps(r) dr\right)^{1/2}\left(\int_0^\eps q_\eps(r) \|v(r)\|_\pi^2dr\right)^{1/2}\le \sqrt{\Xi_\eps}+\lambda_Q^{-1/2}\sqrt{B_\eps(v,v)}.\]
	Therefore \(B_\eps(v,v)\le C\|a\|_\pi(\sqrt{\delta_\eps}+\sqrt{\Xi_\eps})B_\eps(v,v)\). If $B_\eps(v,v)>0$, division by its square root and squaring yield \eqref{eq:corrector-energy}. The zero-energy case is trivial.
	
	Finally,
	\[
	\sup_{r\le2\eps}\norm{v(r)}_\pi
	\le
	\norm{v(0)}_\pi
	+\sqrt{2\eps}\norm{v'}_{L^2} \le \left(C\sqrt{\delta_\eps}+2\sqrt{\eps}\right)\sqrt{B_\eps(v,v)}.
	\]
	Now, using \(\sqrt \eps \le 2^{-1/2}\sqrt{\delta_\eps}\) and
	\eqref{eq:corrector-energy}, we get \eqref{eq:corrector-sup}.
	Cauchy--Schwarz with a factor of \(1\) directly gives \eqref{eq:corrector-L1der}.
\end{proof}

By Lemma~\ref{lem:corrector}, for every \(a\in H_0\) there is a
unique corrector \(v_a^\eps\). Let \(a_1, a_2 \in H_0\) and \(w:=\alpha v_{a_1}^\eps+\beta v_{a_2}^\eps\). Then \(w\) solves \(\frac12(w)''(r)
+q_\eps(r)Qw(r)
=
-r q_\eps(r)Q(\alpha a_1+\beta a_2), \ 0 <r<2\eps\) with \(w(0)=-\alpha a_1-\beta a_2\) and \(w(2\eps)=0\). Therefore, the solution
map \(a\mapsto v_a^\eps\) is linear. Hence, for each
\(r\in[0,2\eps]\), we may define the linear operator
\[
K_\eps(r):H_0\to H_0,
\qquad
K_\eps(r)a:=v_a^\eps(r).
\]

\begin{lemma}
	\label{lem:corrector-operator}
	Let \(\eta_\eps:=\delta_\eps+\Xi_\eps\). Then
	there is $C<\infty$, independent of $\eps$, such that
	\begin{gather}
		\int_0^{2\eps}\norm{K_\eps'(r)}_{\operatorname{op},\pi}^2\,dr
		\le C\eta_\eps,
		\label{eq:Kop-der2}\\
		\int_0^{2\eps}\norm{K_\eps'(r)}_{\operatorname{op},\pi}\,dr
		\le C\sqrt{\eps\eta_\eps},
		\label{eq:Kop-der1}\\
		\sup_{0\le r\le2\eps}\norm{K_\eps(r)}_{\operatorname{op},\pi}
		\le
		C\sqrt{\delta_\eps}
		\bigl(\sqrt{\delta_\eps}+\sqrt{\Xi_\eps}\bigr),
		\label{eq:Kop-sup}\\
		\int_0^\eps
		q_\eps(r)
		\sup_{\norm{a}_\pi\le1}\cD_Q(K_\eps(r)a)\,dr
		\le C\eta_\eps.
		\label{eq:Kop-jump}
	\end{gather}
	Consequently, if $a:\Gamma\to\mathsf H_0$ is bounded and \(v^\eps(y,r):=v^\eps_{a(y)}(r)\), then
	\begin{gather}
		\int_0^{2\eps}
		\sup_{y\in\Gamma}
		\norm{\partial_rv^\eps(y,r)}_\pi^2\,dr
		\le
		C\norm{a}_{L^\infty(\Gamma;\pi)}^2\eta_\eps,
		\label{eq:v-uniform-der2}\\
		\int_0^{2\eps}
		\sup_{y\in\Gamma}
		\norm{\partial_rv^\eps(y,r)}_\pi\,dr
		\le
		C\norm{a}_{L^\infty(\Gamma;\pi)}
		\sqrt{\eps\eta_\eps},
		\label{eq:v-uniform-der1}\\
		\int_0^\eps q_\eps(r)
		\sup_{y\in\Gamma}\cD_Q(v^\eps(y,r))\,dr
		\le
		C\norm{a}_{L^\infty(\Gamma;\pi)}^2\eta_\eps.
		\label{eq:v-uniform-jump}
	\end{gather}
\end{lemma}

\begin{proof}
	Let
	$e_1,\dots,e_{k-1}$ be a $\pi$-orthonormal basis of $\mathsf H_0$.  Since the
	operator norm is bounded by the Hilbert--Schmidt norm,
	\[
	\norm{K_\eps'(r)}_{\operatorname{op},\pi}^2
	\le
	\sum_{\ell=1}^{k-1}
	\norm{K_\eps'(r)e_\ell}_\pi^2=\sum_{\ell=1}^{k-1}\|(v_{e_\ell}^\eps)'(r)\|_\pi^2.
	\]
	Integrating and applying \eqref{eq:corrector-energy} to each \(v_{e_\ell}^\eps\) gives
	\[\int_0^{2\eps}\norm{K_\eps'(r)}_{\operatorname{op},\pi}^2\,dr \le \sum_{\ell=1}^{k-1}\int_0^{2\eps}\|(v_{e_\ell}^\eps)'(r)\|_\pi^2\, dr \le C(k-1)\eta_\eps.\]
	Now, equation \eqref{eq:Kop-der1} follows from Cauchy--Schwarz with a factor of \(1\).
	Likewise, \eqref{eq:corrector-sup} applied to the basis vectors and finite
	dimensionality give \eqref{eq:Kop-sup}.
	
	For \eqref{eq:Kop-jump}, note that $\cD_Q$ is a nonnegative quadratic form on
	$\mathsf H_0$.  Let $S$ be its positive self-adjoint representing operator with
	respect to $\ip{\cdot}{\cdot}_\pi$, so that
	$\cD_Q(u)=\norm{S^{1/2}u}_\pi^2$.  Then
	\[
	\sup_{\norm{a}_\pi\le1}\cD_Q(K_\eps(r)a)
	=
	\norm{S^{1/2}K_\eps(r)}_{\operatorname{op},\pi}^2\le \sum_{\ell=1}^{k-1}\norm{S^{1/2}K_\eps(r)e_\ell}_{\pi}^2
	\le
	\sum_{\ell=1}^{k-1}
	\cD_Q(K_\eps(r)e_\ell).
	\]
	Integrating against $q_\eps(r)\,dr$ and using
	\eqref{eq:corrector-energy} again proves \eqref{eq:Kop-jump}.
	
	Finally $v^\eps(y,r)=K_\eps(r)a(y)$.  The three displayed consequences follow
	immediately from the operator bounds.
\end{proof}

\begin{corollary}
	\label{cor:corrector-parameter}
	Let $a:\Gamma\to\mathsf H_0$ be smooth and put
	$v^\eps(y,r):=v^\eps_{a(y)}(r)$.  For every tangential differential operator
	$D_y:=\sum_{|\alpha|\le m}c_\alpha(y)\partial_y^\alpha$, the estimates
	\eqref{eq:v-uniform-der2}--\eqref{eq:v-uniform-jump} hold with
	$v^\eps$ replaced by $D_yv^\eps$ and $a$ replaced by $D_ya$.
\end{corollary}

\begin{proof}
	We write \(D_y v^\eps(y,r)= K_\eps(r)D_ya(y)\).  Hence
	$D_yv^\eps$ is exactly the corrector associated with $D_ya$.  Apply
	\Cref{lem:corrector-operator}.
\end{proof}

\begin{remark}[Why the first moment disappears]
	If one applies the generator directly to a smooth common-trace function, then
	\[
	Qf(y,r)=rQa(y)+O(r^2),
	\]
	so an $O(\int r q_\eps)$ term appears.  The inhomogeneous term
	$-r q_\eps Qa$ in \eqref{eq:corrector-ode} cancels this contribution exactly.
	Thus no assumption on the first moment
	$\int_0^\eps r q_\eps(r)\,dr$ is needed.  The only uncancelled Taylor remainder is quadratic and is controlled by $\Xi_\eps$.
\end{remark}

\subsection{Corrected Kirchhoff test functions}

Fix $f\in\cC_K$ and let
\(
\phi(y):=f_i(y,0),
\
a_i(y):=\partial_{\nu_i}f_i(y,0).
\)
The common trace makes $\phi$ independent of $i$, while the Kirchhoff condition says
\(
\sum_i\pi_i a_i(y)=0,
\)
so $a(y)\in\mathsf H_0$.  Taylor expansion in Fermi coordinates gives
\[
f(y,r)
=
\phi(y)\bone+r a(y)+r^2 b(y,r),
\]
where $b$ is smooth and uniformly bounded together with its tangential derivatives.

For $0\le r\le2\eps$, set
\(
v^\eps(y,r):=v^\eps_{a(y)}(r).
\)
Tangential derivatives are controlled by \Cref{cor:corrector-parameter}.
Choose a fixed $\zeta\in C^\infty([0,r_0);[0,1])$ with
\(
\zeta=1\text{ on }[0,r_0/3],
\) and \(
\zeta=0\text{ on }[r_0/2,r_0).
\)
Since $q_\eps=0$ on $(\eps,2\eps]$, equation
\eqref{eq:corrector-ode} gives $(v^\eps)''=0$ there; together with
$(v^\eps)'(2\eps)=0$ this implies
\(
(v^\eps)'=0
\text{ on }[\eps,2\eps].
\)
Define on page $i$
\[
\psi_i^\eps(y,r)
:=
\begin{cases}
	v_i^\eps(y,r),&0\le r\le2\eps,\\[1mm]
	\zeta(r)v_i^\eps(y,2\eps),&r>2\eps
\end{cases}
\]
in the collar, and set $\psi_i^\eps=0$ outside the support of $\zeta$.
Then
\begin{equation}
	f_i^\eps:=f_i+\psi_i^\eps
	\label{eq:corrected-test}
\end{equation}
satisfies
\begin{equation}
	\norm{f^\eps-f}_\infty\longrightarrow0,
	\label{eq:test-uniform}
\end{equation}
by finite-dimensional norm equivalence.

\begin{lemma}
	\label{lem:corrected-admissible}
	For every fixed $0<\eps<r_0/6$ and every $f\in\cC_K$, the corrected
	function $f^\eps$ is pagewise $C^2$ and satisfies
	\(
	\partial_{\nu_i}f_i^\eps|_\Gamma=0.
	\)
\end{lemma}

\begin{proof}
	Because $q_\eps$ is bounded and continuous, \(Q\) and \(a(y)\) are constant in \(r\), the corrector
	equation \eqref{eq:corrector-ode} implies
	$v^\eps(y,\cdot)\in C^2$ in the normal variable; smooth dependence on $y$
	follows from linearity in the smooth parameter $a(y)$.  The extension at
	$r=2\eps$ is $C^2$ because $q_\eps=0$ on $(\eps,2\eps]$ and
	$(v^\eps)'=0$ there, while the fixed cutoff $\zeta$ is identically one in a
	neighborhood of $2\eps$ for $\eps<r_0/6$.  The boundary condition
	$(v^\eps)'(0)=-a$ gives \(
	\partial_{\nu_i}f_i^\eps|_\Gamma
	=
	a_i+(v_i^\eps)'(0)=0.
	\)
\end{proof}

Lemma \ref{lem:corrected-admissible} establishes that \(f^\eps\) is an admissible smooth test function for the folded
microscopic generator \(A_\eps\), and Proposition~\ref{prop:folded-generator} gives
\[
M_t^{\eps,f}:=f^\eps(Y_t^\eps)-f^\eps(Y_0^\eps)
-\int_0^tA_\eps f^\eps(Y_s^\eps)\,ds
\]
as a local martingale. For every fixed \(\eps>0\) and \(T<\infty\), both
\(f^\eps\) and \(A^\eps f^\eps\) are bounded on the finite union of compact pages. Hence
\[
\sup_{t\le T}|M_t^{\eps,f}|
\le
2\|f^\eps\|_\infty
+T\|A^\eps f^\eps\|_\infty
<\infty.
\]
Thus \(M^{\eps,f}\) is a true martingale on every finite time interval.

\section{Convergence from deterministic starting points}
\label{sec:main-limit}

The next proposition is the analytic core of the convergence proof.

\begin{proposition}
	\label{prop:generator-convergence}
	Let $f\in\cC_K$ and let $f^\eps$ be the corrected test function
	\eqref{eq:corrected-test}.  Under Assumption \ref{ass:main-scaling}, for every $T<\infty$,
	\begin{equation}
		\sup_{\tilde x\in\widetilde E}
		\E_{\tilde x}
		\int_0^T
		\left|
		A_\eps f^\eps(Y_s^\eps)
		-
		A_0f(Z_s^\eps)
		\right|\,ds
		\longrightarrow0.
		\label{eq:integrated-generator}
	\end{equation}
%
\end{proposition}

\begin{proof}
	We split the collar into four regions.
	
	\smallskip
	\noindent
	\emph{Region I: $0\le r\le\eps$.} On this region \(f^\eps_i=f_i+v_i^\eps\)
	Taylor expansion gives
	\[
	f(y,r)=\phi(y)\bone+r a(y)+r^2b(y,r).
	\]
	Since the \(f_i\)'s are smooth up to the boundary, $b$ and its first two tangential derivatives uniformly bounded.  Therefore
	\[
	q_\eps(r)Qf(y,r)
	=
	rq_\eps(r)Qa(y)
	+
	r^2q_\eps(r)Qb(y,r).
	\]
	Using the corrector equation \eqref{eq:corrector-ode} and the Fermi coordinate decomposition \(\Delta_{g_i^\eps}=\partial_rr+\kappa_i^\eps\partial_r+\Delta_{h_i^\eps(r)}\), on page $i$ we obtain
	\begin{equation*}
		\begin{array}{r c c c c c c c c}
			A_\eps f^\eps-Af
			&=&
			\frac12(\Delta_{g_i^\eps}-\Delta_{g_i})f_i
			&+&
			r^2q_\eps(r)(Qb)_i
			&+&
			\frac12\kappa_i^\eps\partial_rv_i^\eps
			&+&
			\frac12\Delta_{h_i^\eps(r)}v_i^\eps
			\\[1ex]
			& =: &
			S_1^\eps
			&+&
			S_2^\eps
			&+&
			S_3^\eps
			&+&
			S_4^\eps .
		\end{array}
		\label{eq:error-region-I}
	\end{equation*}
	\(S_1^\eps\) is supported in $0\le r<2\rho_\eps$. By \eqref{eq:fermi-uniform-coefficients}, the coefficients
	$\kappa_i^\eps$ and the operators $\Delta_{h_i^\eps(r)}$ are uniformly
	bounded on the fixed collar. Since the metrics \(h_i\) are smooth on \(\Gamma \times [0,r_0]\), \(\kappa_i\) and \(\Delta_{g_i}\) are bounded as well.
	Since $f$ is fixed and smooth,
	\begin{equation*}
		\sup_{\substack{0<\eps<\eps_0,\\1\le i\le k}}\sup_{\substack{y \in \Gamma,\\r<2\rho_\eps}}
		\abs{S_1^\eps}=\sup_{\substack{0<\eps<\eps_0,\\1\le i\le k}}\sup_{\substack{y \in \Gamma,\\r<2\rho_\eps}}\abs{(\kappa_i^\eps-\kappa_i)\partial_rf_i}+\abs{(\Delta_{h_i^\eps(r)}-\Delta_{h_i(r)})f_i}
		\le C.
		\label{eq:metric-error-bounded}
	\end{equation*}
	Thus \eqref{eq:occupation-L1} gives
	\[\sup_{\tilde x}\E_{\tilde x}\int_0^T\abs{S_1^\eps(Y_s^\eps)}\,ds \le C_T \int_0^{r_0/2}\mathbf{1}_{0\le r<2\rho_\eps}\le C_T \rho_\eps\to 0.\]
	For the \(S_2^\eps\), \((Qb)_i\) is bounded and using \eqref{eq:occupation-L1} with \(\phi(r)=r^2q_\eps(r)\mathbf{1}_{r \le \eps}\), we get
	\[
	\sup_{\tilde x}
	\E_{\tilde x}\int_0^T
	R_s^{\eps\,2}q_\eps(R_s^\eps)
	\mathbf1_{\{R_s^\eps\le\eps\}}\,ds
	\le
	C_T\Xi_\eps\to 0.
	\label{eq:taylor-rem-bound}
	\]
	For \(S_3^\eps\), \eqref{eq:fermi-uniform-coefficients} gives \(\kappa_i^\eps\le C\). Therefore, finite-dimensional norm equivalence, \eqref{eq:occupation-L1} and
	\eqref{eq:v-uniform-der1} give
	\begin{align*}
		&\sup_{\tilde x}
		\E_{\tilde x}\int_0^T
		\abs{
			c_{I_s^\eps}^\eps
			\partial_rv_{I_s^\eps}^\eps
			(Y_s^\eps,R_s^\eps)}
		\mathbf1_{\{R_s^\eps\le\eps\}}\,ds\le
		C_T\int_0^\eps
		\sup_{y}\norm{\partial_rv^\eps(y,r)}_\pi\,dr
		\le
		C_T\sqrt{\eps\eta_\eps}
		\to0.
		\label{eq:curvature-error-bound}
	\end{align*}
	Finally for \(S_4^\eps\), by \Cref{lem:corrector-operator,cor:corrector-parameter},
	\[
	\begin{aligned}
		\sup_{y\in\Gamma,r\le2\eps}
		&\left(
		\norm{v^\eps(y,r)}_\pi
		+
		\norm{\nabla_\Gamma v^\eps(y,r)}_\pi
		+
		\norm{\nabla_\Gamma^2v^\eps(y,r)}_\pi
		\right)\\&\quad=\sup_{0<r\le2\eps}\|K(r)\|_{\operatorname{op},\pi}\sup_{y \in \Gamma}
		\left(
		\norm{a(y)}_\pi
		+
		\norm{\nabla_\Gamma a(y)}_\pi
		+
		\norm{\nabla_\Gamma^2a(y)}_\pi
		\right)
		\to0.
	\end{aligned}
	\]
	Together with \eqref{eq:fermi-uniform-coefficients} and finite-dimensional norm equivalnce, this gives \[\sup_{\substack{1\le i \le k, \\y \in \Gamma, 0<r\le \eps}} |S_4^\eps| \le C\sup_{\substack{1\le i \le k, \\ 0<r\le \eps}} \|v_i^\eps(\cdot,r)\|_{C^2(\Gamma,h)}\to 0.\]
	
	\smallskip
	\noindent
	\emph{Region II: $\eps<r\le2\eps$.}
	Here $q_\eps=0$.  Since
	$2\rho_\eps<\eps$, we have $g_i^\eps=g_i$ throughout this region.  Moreover
	$\partial_r v_i^\eps=0$ on $[\eps,2\eps]$, and hence
	\[
	\sup_{\substack{1\le i \le k, y \in \Gamma,\\\eps<r\le 2\eps}} |A_\eps f^\eps-Af|
	=
	\sup_{\substack{1\le i \le k, y \in \Gamma,\\\eps<r\le  2\eps}}\abs{\frac12\Delta_{h_i(r)}v_i^\eps}\le C\sup_{\substack{1\le i \le k, \\\eps<r\le  2\eps}}\|v_i(\cdot,r)\|_{C^2(\Gamma,h)}\to 0.
	\]
	
	\smallskip
	\noindent
	\emph{Region III: $2\eps<r<r_0/2$.}
	In this region, we have
	\(
	A_\eps f^\eps-Af=1/2 \, \Delta_{h_i(r)}	\psi_i^\eps.
	\)
	Moreover, \[\Delta_{g_i}\psi_i^\eps=\zeta''(r)v_i^\eps(y,2\eps)+\kappa_i^\eps(y,r)\zeta'(r)v_i^\eps(y,2\eps)+\zeta(r)\Delta_{h_i(r)}v_i^\eps(\cdot,2\eps)(y).\]
	Consequently
	\begin{equation*}
		\sup_{\substack{1\le i \le k, y \in \Gamma,\\\eps<r\le 2\eps}}
		\abs{A_\eps f^\eps-Af}\le\sup_{\substack{1\le i \le k, y \in \Gamma,\\\eps<r\le 2\eps}}
		\abs{\Delta_{g_i}\psi_i^\eps(y,r)}\le C\sup_{\substack{1\le i \le k, \\\eps<r\le  2\eps}}\|v_i(\cdot,r)\|_{C^2(\Gamma,h)}\to 0.
		\label{eq:outer-corrector}
	\end{equation*}
	
	\smallskip
	\noindent
	\emph{Region IV: outside the support of $\zeta$.}
	There $f^\eps=f$ and, for small $\eps$, $g_i^\eps=g_i$, so the generator error is
	identically zero.
	
	Combining the four regions proves
	\eqref{eq:integrated-generator}.

\end{proof}

We may now prove our main result.

\begin{theorem}
	\label{thm:main}
	Suppose Assumption \ref{ass:main-scaling} holds.  Let
	$\tilde x_\eps\in\widetilde E$ satisfy
	\(
	\mathfrak q(\varpi(\tilde x_\eps))
	\to z\in\cM.
	\)
	Then, for every $T<\infty$,
	\begin{equation}
		Z^\eps
		\Longrightarrow
		Z
		\qquad\text{in }D([0,T],\cM),
		\label{eq:main-convergence}
	\end{equation}
	where $Z$ is the unique solution of the $(A_0,\mathcal C_K,\delta_z)$ martingale problem
	of \Cref{prop:target-mp}. Moreover, \(Z\) has continuous sample paths almost surely.
\end{theorem}

\begin{proof}
	By \Cref{lem:C-tightness}, the laws of $Z^\eps$ are C-tight.  Consider any convergent subsequence and denote its limit by $\overline Z$.  The initial condition is $z$ by the assumed convergence of starting points.
	
	Fix $f\in\cC_K$ and the corrected functions $f^\eps$ from
	\eqref{eq:corrected-test}.  By \eqref{eq:test-uniform},
	\(
	\sup_{t\le T}
	\abs{f^\eps(Y_t^\eps)-f(Z_t^\eps)}
	\to0
	\)
	deterministically.  By \Cref{prop:generator-convergence},
	\[
	\int_0^tA_\eps f^\eps(Y_s^\eps)\,ds
	-
	\int_0^tA_0f(Z_s^\eps)\,ds
	\to0
	\]
	in $L^1$, uniformly for $t\le T$.  The microscopic processes
	\(
	M_t^{\eps,f}
	\)
	are martingales, and by \Cref{prop:generator-convergence} they are uniformly square integrable.
	
	Let $0\le s<t\le T$ and let
	$G$ be any bounded continuous functional of the path up to time $s$.  Since
	$M^{\eps,f}$ is a martingale and $G(Z^\eps_{\cdot\wedge s})$ is measurable with
	respect to the microscopic filtration at time $s$,
	\begin{equation}
		\E_{\tilde{x}_\eps}\left[
		G(Z^\eps_{\cdot\wedge s})
		\bigl(M_t^{\eps,f}-M_s^{\eps,f}\bigr)
		\right]=0.
		\label{eq:prelimit-martingale-identity}
	\end{equation}
	Defining \(H_{s,t}^\eps:=f(Z_t^\eps)-f(Z_s^\eps)-\int_s^t A_0 f(Z_u^\eps)du\), equations \eqref{eq:test-uniform} and \eqref{eq:integrated-generator} give
	\[\E_{\tilde{x}_\eps}\abs{M_t^{\eps,f}-M_s^{\eps,f}-H_{s,t}^\eps}\le2\|f^\eps-f\|_\infty+\E_{\tilde{x}_\eps}\int_0^T\abs{A_\eps f^\eps(Y_u^\eps)-A_0f(Z_u^\eps)}\,du \to 0,\]
	which in turn implies that
	\(
	\E_{\tilde x_\eps}[G(Z^\eps_{\cdot\wedge s})H_{s,t}^\eps]\to0.
	\)
	Because $f\in\cC_K$, both $f$ and $A_0f$ are continuous and bounded on the
	compact quotient $\cM$.  Since the limiting law is supported on \(C([0,T],M)\) and \(J_1\)-convergence to a continuous path implies uniform convergence,
	the map
	\[
	\omega\longmapsto
	G(\omega_{\cdot\wedge s})
	\left(
	f(\omega(t))-f(\omega(s))
	-\int_s^t A_0f(\omega(u))\,du
	\right)
	\]
	is continuous at every limiting path in the Skorokhod $J_1$ topology. Moreover, it is bounded, since \(G\), \(f\), and \(A_0f\)
	are bounded. Hence, by the continuous mapping theorem,
	\[
	G(Z^\eps_{\cdot\wedge s})H_{s,t}^\eps
	\Rightarrow
	G(Z_{\cdot\wedge s})
	\left(
	f(Z_t)-f(Z_s)-\int_s^t A_0f(Z_u)\,du
	\right).
	\]
	Boundedness upgrades convergence in distribution to convergence in \(L^1\). Therefore
	\[
	\E_z\left[
	G(\overline Z_{\cdot\wedge s})
	\left(
	f(\overline Z_t)-f(\overline Z_s)
	-
	\int_s^tAf(\overline Z_u)\,du
	\right)
	\right]=0.
	\]
	Thus $\overline Z$ solves the martingale problem for $(A_0,\cC_K)$ started
	from $z$.  By Ethier--Kurtz, Theorem~4.4.1, as verified in
	\Cref{prop:target-mp}, that martingale problem is well posed.  Therefore every subsequential limit has the same law, namely the law of the Kirchhoff diffusion $Z$ started at $z$. By C-tightness such law is
	supported on $C([0,T],\cM)$ for every $T$.
\end{proof}
%

\begin{remark}
	When \(z\in\Gamma\), no convergence of the microscopic
	page labels is required: any sequence of starting points whose glued
	images converge to \(z\) has the same limiting law.
\end{remark}

\begin{remark}[Curvature universality]
	The limiting interface law is independent of the second fundamental
	forms of the pages. Indeed, the first normal variation of the Fermi
	volume density enters the microscopic generator through
	\(
	\kappa_i^\eps(y,r)\)
	whose boundary value is the mean-curvature coefficient up to the
	chosen normal convention. The corresponding corrector contribution
	\(\kappa_i^\eps\partial_r v_i^\eps\) vanishes in the limit by the
	estimate established in Proposition~\ref{prop:generator-convergence}.
	Thus curvature remains only through the bulk operators
	\(\Delta_{g_i}\) and produces no additional leading-order interface
	term.
\end{remark}

\appendix
\section{Elliptic bootstrap in the proof of \ref{lem:weak-smooth}}
\label{app:bootstrap}

Here we discuss the bootstrap mechanism that completes the proof of Lemma \ref{lem:weak-smooth}.

Fix an integer $m\ge1$ and a smaller half-cylinder
\(
Q_2:=U_2\times(0,r_2),\) with \(
U_2\Subset U_0\) and \( 0<r_2<r_1/2.
\)
Since $U_2\Subset U_1$, choose
$U_2\Subset U_\ast\Subset U_1$ and $h_m>0$ so small that every cumulative
tangential translation arising from at most $m$ difference quotients with
$|h_j|<h_m$ sends $U_\ast$ into $U_1$.  Choose a cutoff \(\eta \in C_c^\infty(U_\ast \times [0,r_1/4))\) equal to one on $U_2\times [0,r_2]$. We prove by induction on $m$ that, for every tangential
multi-index $\beta$ with $|\beta|\le m$,
\(
D_y^\beta u_i\in H^1(Q_2)
\)
and
\begin{equation}
	\sum_i\sum_{|\beta|\le m}
	\|D_y^\beta u_i\|_{H^1(Q')}
	\le
	C_{m,Q'}\left(
	\sum_i\|u_i\|_{H^1(U_1\times(0,r_1))}
	+
	\sum_i\|G_i\|_{H^m(U_1\times(0,r_1))}
	\right).
	\label{eq:all-tangential-H1}
\end{equation}
Here and below $D_y^\beta$ denotes an actual weak tangential derivative;
difference quotients are used only to obtain these derivatives.

The case $m=1$ is precisely the difference-quotient argument above.  Assume
now that the assertion has been proved through tangential order $m-1$.	
Applying discrete integration by parts successively and the discrete Leibniz rule gives a higher dimensional version of \eqref{eq:DQ-identity}. In particular, the product rule separates, in each differentiated principal term,
a unique top-order contribution from lower-order coefficient commutators.
Schematically,
\begin{equation}
	\Delta_{\mathbf h}
	\bigl(B_i\nabla u_i\bigr)
	=
	(\tau_{\mathbf h}B_i)\,
	\nabla\Delta_{\mathbf h}u_i
	+
	\mathcal R_{B,i}^{(m)},
	\label{eq:higher-DQ-Leibniz}
\end{equation}
where $\tau_{\mathbf h}B_i$ denotes a translate of $B_i$, and
$\mathcal R_{B,i}^{(m)}$ is a finite sum of terms of the form
\(
(\tau D^{\gamma}_{\mathbf h}B_i)\,
\nabla D^{\delta}_{\mathbf h}u_i,\) with\(
|\gamma|\ge1\) and \(|\delta|\le m-1.
\)
Analogous expansions hold for $J_i u_i$ and $J_iG_i$.

Because $B_i$ is smooth, all iterated difference quotients of its coefficients
of order at most $m$ are uniformly bounded for sufficiently small
$\mathbf h$; the same is true for $J_i$.  Moreover, every derivative of
$u_i$ occurring in the remainder
$\mathcal R_{B,i}^{(m)}$ has tangential order at most $m-1$ and is therefore
controlled by the induction hypothesis.  The forcing terms are controlled by
the standard difference-quotient estimates applied to $J_iG_i$, and hence by
$\|G_i\|_{H^m}$.  On the other hand, the first term on the right-hand side of
\eqref{eq:higher-DQ-Leibniz} contains the new highest-order quantity
$\nabla\Delta_{\mathbf h}u_i$.  Since translations preserve the uniform
ellipticity bounds for $B_i$, it produces the same coercive term as in the
first-order argument, that is
\(
c\sum_i
\|\eta\nabla\Delta_{\mathbf h}u_i\|_{L^2}^2.
\)
Estimating the cutoff and commutator terms by Cauchy--Schwarz and Young's
inequality, and absorbing the resulting small multiples of the coercive term,
therefore gives
\[
\sum_i
\|\Delta_{\mathbf h}u_i\|_{H^1(Q_2)}^2
\le
C_{m,Q_2}\left(
\sum_i\|u_i\|_{H^1(Q_1)}^2
+
\sum_i\|G_i\|_{H^m(Q_1)}^2
\right),
\]
uniformly as \(h_1,\ldots,h_m\) tend to zero.  Letting
$h_1,\ldots,h_m$ tend successively to zero and using weak compactness in
$H^1(Q_2)$, exactly as in the first-order argument, yields
\(
D_y^\beta u_i\in H^1(Q_2)\) for every tangential \(\beta,\, |\beta|=m.
\)
The continuity of the trace map
$H^1(Q_2)\to H^{1/2}(U_2)$ also shows that the traces
$\Tr D_y^\beta u_i$ remain common to all pages.  This closes the tangential
induction and proves \eqref{eq:all-tangential-H1}.

Thus derivatives with arbitrarily high tangential order and with at most one
normal derivative are locally in $L^2$ up to the binding.  It remains to
recover derivatives containing two or more normal indices.  We use
\eqref{eq:recover-rr},
recursively.

More precisely, we prove by induction on $\ell\ge1$ that
\(
\partial_r^jD_y^\beta u_i\in L^2_{\mathrm{loc}}\) for every \(0\le j\le\ell
\)
and every tangential multi-index $\beta$.  The case $\ell=1$ follows from
\eqref{eq:all-tangential-H1}, while \eqref{eq:recover-rr} gives the case
$j=2$.  Suppose now that all derivatives with at most $\ell$ normal
derivatives have already been obtained.  Apply
$\partial_r^{\ell-1}D_y^\beta$ to \eqref{eq:recover-rr} in the sense of
distributions.  The left-hand side is
\(
\partial_r^{\ell+1}D_y^\beta u_i.
\)
Every term on the right contains at most $\ell$ normal derivatives of $u_i$ and
is controlled by the tangential estimate and the induction hypothesis.  The same argument applies to the remaining lower-order terms.
Since the coefficients are smooth and bounded together with all derivatives
on compact subcylinders of the collar, and since $G_i$ is smooth, the entire
right-hand side belongs to $L^2_{\mathrm{loc}}$ closing the induction on the number of normal derivatives.

It follows that
\(
u_i\in H^s_{\mathrm{loc}}
(U_0\times[0,r_1/2))\) for every \(s\).
A finite covering of $\Gamma$, together with ordinary interior elliptic
regularity away from the binding, then yields
\(
u_i\in H^s(M_i)\) for every \(s
\),
and therefore
\(
u_i\in C^\infty(M_i)
\)
by Sobolev embedding.

\section{Localization in the proof of \Cref{lem:C-tightness}}

Choose finitely many coordinate neighborhoods
\[
\mathcal V_\ell\Subset \mathcal U_\ell,
\qquad \ell=1,\ldots,N,
\]
covering \(M\), where each \(\mathcal U_\ell\) is either an interface
neighborhood of the form considered above or an interior chart whose
closure is disjoint from \(\Gamma\).  The neighborhoods may be chosen
with a uniform positive buffer: there exists \(d_0>0\) such that, for
every \(x\in M\), one can find \(\ell\) with
\[
x\in\mathcal V_\ell,
\qquad
d_M(x,M\setminus\mathcal U_\ell)\ge d_0.
\]
On the compact closure of each chart, the coordinate distance is
uniformly equivalent to \(d_M\).  Hence the stopped coordinate estimates
proved above imply that, for every \(\gamma>0\),
\[
\lim_{\delta\downarrow0}
\sup_{0<\varepsilon\le\varepsilon_0}
\sup_{\tilde x\in\widetilde E}
\sup_{\tau\le T}
\mathbb P_{\tilde x}
\left(
d_M\!\left(
\bar Z^\varepsilon_{(\tau+\delta)\wedge
	\sigma^\varepsilon_\ell\wedge T},
\bar Z^\varepsilon_\tau
\right)>\gamma
\right)
=0,
\]
whenever \(\bar Z^\varepsilon_\tau\in\mathcal V_\ell\), where
\(\sigma^\varepsilon_\ell\) denotes the first exit time after \(\tau\)
from \(\mathcal U_\ell\).  Since leaving \(\mathcal U_\ell\) from a
point of \(\mathcal V_\ell\) requires a displacement of at least
\(d_0\), the same estimate with \(\gamma=d_0\) shows that
\[
\lim_{\delta\downarrow0}
\sup_{0<\varepsilon\le\varepsilon_0}
\sup_{\tilde x\in\widetilde E}
\sup_{\tau\le T}
\mathbb P_{\tilde x}
\bigl(
\sigma^\varepsilon_\ell\le(\tau+\delta)\wedge T
\bigr)
=0 .
\]
Using the finite cover and partitioning according to a chart
\(\mathcal V_\ell\) containing \(\bar Z^\varepsilon_\tau\), we therefore
obtain the global Aldous condition
\[
\lim_{\delta\downarrow0}
\sup_{0<\varepsilon\le\varepsilon_0}
\sup_{\tilde x\in\widetilde E}
\sup_{\tau\le T}
\mathbb P_{\tilde x}
\left(
d_M\!\left(
\bar Z^\varepsilon_{(\tau+\delta)\wedge T},
\bar Z^\varepsilon_\tau
\right)>\gamma
\right)
=0
\]
for every \(\gamma>0\).

\bibliographystyle{plain}
\bibliography{biblio_clean}

@incollection{Walsh1978,
	author    = {Walsh, John B.},
	title     = {A diffusion with a discontinuous local time},
	booktitle = {Temps locaux},
	series    = {Ast\'erisque},
	number    = {52--53},
	pages     = {37--45},
	publisher = {Soci\'et\'e Math\'ematique de France},
	year      = {1978},
	url       = {https://www.numdam.org/item/AST_1978__52-53__37_0/}
}

@incollection{BarlowPitmanYor1989,
	author    = {Barlow, Martin T. and Pitman, Jim and Yor, Marc},
	title     = {On {Walsh's} Brownian motions},
	booktitle = {S\'eminaire de Probabilit\'es XXIII},
	series    = {Lecture Notes in Mathematics},
	volume    = {1372},
	pages     = {275--293},
	publisher = {Springer},
	address   = {Berlin},
	year      = {1989},
	doi       = {10.1007/BFb0083979},
	url       = {https://www.numdam.org/item/SPS_1989__23__275_0/}
}

@article{FreidlinWentzell1993,
	author  = {Freidlin, Mark I. and Wentzell, Alexander D.},
	title   = {Diffusion processes on graphs and the averaging principle},
	journal = {Annals of Probability},
	volume  = {21},
	number  = {4},
	pages   = {2215--2245},
	year    = {1993},
	doi     = {10.1214/aop/1176989018}
}

@article{FreidlinSheu2000,
	author  = {Freidlin, Mark and Sheu, Shuenn-Jyi},
	title   = {Diffusion processes on graphs: stochastic differential equations, large deviation principle},
	journal = {Probability Theory and Related Fields},
	volume  = {116},
	number  = {2},
	pages   = {181--220},
	year    = {2000},
	doi     = {10.1007/PL00008726}
}

@article{KostrykinPotthoffSchrader2012,
	author  = {Kostrykin, Vadim and Potthoff, J\"urgen and Schrader, Robert},
	title   = {Brownian motions on metric graphs},
	journal = {Journal of Mathematical Physics},
	volume  = {53},
	number  = {9},
	pages   = {095206},
	year    = {2012},
	doi     = {10.1063/1.4714661}
}

@article{BrinKifer2001,
	author  = {Brin, Michael and Kifer, Yuri},
	title   = {Brownian motion, harmonic functions and hyperbolicity for {Euclidean} complexes},
	journal = {Mathematische Zeitschrift},
	volume  = {237},
	number  = {3},
	pages   = {421--468},
	year    = {2001},
	doi     = {10.1007/PL00004875}
}

@article{Bouziane2005,
	author  = {Bouziane, Taoufik},
	title   = {Brownian motion in {Riemannian} admissible complexes},
	journal = {Illinois Journal of Mathematics},
	volume  = {49},
	number  = {2},
	pages   = {559--580},
	year    = {2005},
	doi     = {10.1215/ijm/1258138035}
}

@article{NyeWhite2014,
	author  = {Nye, T. M. W. and White, M. C.},
	title   = {Diffusion on some simple stratified spaces},
	journal = {Journal of Mathematical Imaging and Vision},
	volume  = {50},
	number  = {1--2},
	pages   = {115--125},
	year    = {2014},
	doi     = {10.1007/s10851-013-0457-0}
}

@article{AlbeverioKusuoka2012,
	author  = {Albeverio, Sergio and Kusuoka, Seiichiro},
	title   = {Diffusion processes in thin tubes and their limits on graphs},
	journal = {Annals of Probability},
	volume  = {40},
	number  = {5},
	pages   = {2131--2167},
	year    = {2012},
	doi     = {10.1214/11-AOP667}
}

@article{HairerManson2010,
	author  = {Hairer, Martin and Manson, Charles},
	title   = {Periodic homogenization with an interface: the one-dimensional case},
	journal = {Stochastic Processes and their Applications},
	volume  = {120},
	number  = {8},
	pages   = {1589--1605},
	year    = {2010},
	doi     = {10.1016/j.spa.2010.03.016}
}

@article{HairerManson2011,
	author  = {Hairer, Martin and Manson, Charles},
	title   = {Periodic homogenization with an interface: the multi-dimensional case},
	journal = {Annals of Probability},
	volume  = {39},
	number  = {2},
	pages   = {648--682},
	year    = {2011},
	doi     = {10.1214/10-AOP564}
}

@book{ethier2009markov,
  author    = {Ethier, Stewart N. and Kurtz, Thomas G.},
  title     = {Markov Processes: Characterization and Convergence},
  series    = {Wiley Series in Probability and Mathematical Statistics},
  year      = {1986},
  publisher = {Wiley},
  address   = {New York},
  doi       = {10.1002/9780470316658}
}

@article{FreidlinWentzell2004,
  author  = {Freidlin, Mark I. and Wentzell, Alexander D.},
  title   = {Diffusion processes on an open book and the averaging principle},
  journal = {Stochastic Process. Appl.},
  volume  = {113},
  number  = {1},
  pages   = {101--126},
  year    = {2004},
  doi     = {10.1016/j.spa.2004.03.009}
}

@book{FreidlinWentzell2012,
  author    = {Freidlin, Mark I. and Wentzell, Alexander D.},
  title     = {Random Perturbations of Dynamical Systems},
  edition   = {3},
  series    = {Grundlehren der mathematischen Wissenschaften},
  volume    = {260},
  year      = {2012},
  publisher = {Springer},
  address   = {Heidelberg},
  doi       = {10.1007/978-3-642-25847-3}
}

@article{PivarskiSaloffCoste2008,
  author  = {Pivarski, Melanie and Saloff-Coste, Laurent},
  title   = {Small time heat kernel behavior on {Riemannian} complexes},
  journal = {New York J. Math.},
  volume  = {14},
  pages   = {459--494},
  year    = {2008}
}

@article{BendikovSaloffCosteSalvatoriWoess2011,
  author  = {Bendikov, Alexander and Saloff-Coste, Laurent and Salvatori, Maura and Woess, Wolfgang},
  title   = {The heat semigroup and {Brownian} motion on strip complexes},
  journal = {Adv. Math.},
  volume  = {226},
  number  = {1},
  pages   = {992--1055},
  year    = {2011},
  doi     = {10.1016/j.aim.2010.07.014}
}

@article{AkdumanKuchment2024,
	author        = {Akduman, Setenay and Kuchment, Peter},
	title         = {On open book analogs of quantum graphs},
	journal       = {arXiv preprint arXiv:2406.16190},
	year          = {2024},
	eprint        = {2406.16190},
	archivePrefix = {arXiv},
	primaryClass  = {math-ph}
}

@article{Corbin2020,
	author  = {Corbin, James E.},
	title   = {Convergence of spectra of uniformly fattened open book structures},
	journal = {Journal of Mathematical Analysis and Applications},
	volume  = {489},
	number  = {2},
	pages   = {124181},
	year    = {2020},
	doi     = {10.1016/j.jmaa.2020.124181}
}

@article{ExnerPost2005,
	author  = {Exner, Pavel and Post, Olaf},
	title   = {Convergence of spectra of graph-like thin manifolds},
	journal = {Journal of Geometry and Physics},
	volume  = {54},
	number  = {1},
	pages   = {77--115},
	year    = {2005},
	doi     = {10.1016/j.geomphys.2004.08.005}
}

\end{document}